\documentclass[a4paper, 11pt]{article}
\usepackage[utf8]{inputenc}
\usepackage[a4paper, total={7in, 10in}]{geometry}
\usepackage{mathrsfs}
\usepackage{amsfonts, amsmath, amssymb, amsthm}
\usepackage{url,hyperref}

\numberwithin{equation}{section}

\newcommand{\R}{\mathbb{R}}

\newcommand{\N}{\mathbb{N}}
\newcommand{\etalchar}[1]{$^{#1}$}

\makeatletter
\newsavebox{\@brx}
\newcommand{\llangle}[1][]{\savebox{\@brx}{\(\m@th{#1\langle}\)}
  \mathopen{\copy\@brx\kern-0.5\wd\@brx\usebox{\@brx}}}
\newcommand{\rrangle}[1][]{\savebox{\@brx}{\(\m@th{#1\rangle}\)}
  \mathclose{\copy\@brx\kern-0.5\wd\@brx\usebox{\@brx}}}
\makeatother

\newtheorem{theorem}{Theorem}
\newtheorem{definition}{Definition}
\newtheorem{proposition}{Proposition}
\newtheorem{lemma}{Lemma}
\theoremstyle{remark}
\newtheorem{remark}{Remark}

\title{\textbf{On fractional and classical hyperbolic obstacle-type problems}}
\author{
  Pedro Miguel Campos\footnote{CMAFcIO – Departamento de Matemática, Faculdade de Ciências, Universidade de Lisboa P-1749-016 Lisboa, Portugal \\Email address: \texttt{pmcampos@fc.ul.pt}}\\
  \and
  José Francisco Rodrigues\footnote{CMAFcIO – Departamento de Matemática, Faculdade de Ciências, Universidade de Lisboa P-1749-016 Lisboa, Portugal\\Email address: \texttt{jfrodrigues@ciencias.ulisboa.pt}}
}

\begin{document}

\maketitle

\vspace{-1cm}

\begin{center}
    \textit{Dedicated to Pierluigi Colli in its $65^{\text{th}}$ birthday with friendship.}
\end{center}

\begin{abstract}
    We consider weak solutions for the obstacle-type viscoelastic ($\nu>0$) and very weak solutions for the obstacle inviscid ($\nu=0$) Dirichlet problems for the heterogeneous and anisotropic wave equation in a fractional framework based on the Riesz fractional gradient $D^s$ ($0<s<1$).  We use weak solutions of the viscous problem to obtain very weak solutions of the inviscid problem when $\nu\searrow 0$. We prove that the weak and very weak solutions of those problems in the fractional setting converge as $s\nearrow 1$ to a weak solution and to a very weak solution, respectively, of the correspondent problems in the classical framework.
\end{abstract}

\section{Introduction}
Inequalities in mechanics and physics arise naturally in situations when thresholds are crossed or attained and the equations are subjected to certain constraints \cite{duvaut1976Inequalities}. Typical examples are the displacement of an elastic membrane constrained by one or two an obstacles and the phase transition in materials. While for stationary or diffusion type problems there is a large literature, the hyperbolic problems present serious challenges and have been little studied. In this article we are motivated by the example of the constrained dynamics of a string or a membrane. This type of problems with inequalities for the wave equation was already considered by \cite{amerio1975Study, schatzman1980Concave, jarusek1992Variational,bonfanti2004Convergence, eck2008Unilateral, bock2011Unilateral, bonetti2017On, bonafini2019Variational, omata2022Variational} but they are far from being well understood. In fact, besides the one-dimensional case with concave obstacle, we do not know if the solutions of these problems are unique, and even if they are not unique, if there exists a solution that preserves its initial energy. 

On the other hand, there has been a recent interest in modeling some problems in solid mechanics with nonlocal laws involving fractional space derivatives. In particular, a nonlocal theory, called Peridynamics, was developed by Silling in \cite{silling2000Reformulation, Silling2007Peridynamic}. The nonlocal theories aim to unify the mechanics of continuous and discontinuous media within a simple framework, covering not only the classical problems in elasticity, but also the formation of cracks and long-range forces. In \cite{du2013Nonlocal} a new nonlocal approach, based on the traditional vector calculus was introduced and in the recent survey \cite{du2019Nonlocal} several integral operators are considered, including those arising in models of materials occupying the whole Euclidean space $\R^d$ and whose material points can interact with each other within an infinite horizon, as the fractional Laplacian $(-\Delta)^s$ with $0<s<1$. Just like the classical Laplacian, the fractional Laplacian can also be decomposed as $(-\Delta)^s=-D^s\cdot D^s$, where $D^s$ is the  Riesz fractional gradient, i.e., the only vector-valued operator, up to multiplicative constants, that is translational and rotational invariant, continuous in the sense of distributions and $s$-homogeneous, as it was proven in \cite{silhavy2020Fractional}. These properties of the Riesz fractional gradient are well suited for anisotropic and nonhomogeneous problems, in particular in the linearization of  Eringen’s models for nonlocal elasticity, see \cite{bellido2023eringen}.

In this work we consider a unifying approach that generalizes the existence results of \cite{jarusek1992Variational} and of \cite{bonetti2017On}, for the dynamics of the viscoelastic membrane with one obstacle or two obstacles, respectively, to a class of non-homogeneous linear fractional operators $D^s\cdot(A D^s \cdot)$ and $D^s\cdot(B D^s \cdot)$, including the fractional Laplacian $(-\Delta)^s$, $0<s<1$. Extending to the fractional framework the variational methods used in those two papers for the classical case $s=1$ with the Laplacian $\Delta=D\cdot D$, where $D =D^1$ denotes the gradient, we consider the problem
\begin{equation}\label{eq:problem}
    \begin{cases}
        \ddot{u}-D^s\cdot(A D^s u)-\nu D^s\cdot (B D^s \dot u) + \beta(u)\ni g & \mbox{ in } Q_T:=(0,T)\times\Omega,\\
        u=0 & \mbox{ on } (\R^d\setminus\Omega)\times(0,T),\\
        u(0,\cdot)=w_0,\quad \dot{u}(0,\cdot)=w_1 & \mbox{ on } \Omega.
    \end{cases}
\end{equation}
Here $\Omega$ is a bounded subset of $\R^d$, $\nu\geq 0$ denotes the viscosity parameter, $T>0$ is a given final time, $D^s$ is the Riesz fractional gradient (see next section for the definition), and $\beta$ is a maximal monotone graph sufficiently general so that we can cover the cases when there is no obstacle, i.e., the case in which $\beta$ is just a non-decreasing function on $\R$, possibly with discontinuities, as well as the case of unilateral constraints, corresponding to one obstacle or two obstacles, respectively, when the domain of $\beta$ is a half-line or a bounded interval in $\R$, containing $0$. In addition to the proof of existence of weak solutions to \eqref{eq:problem} when $\nu>0$, we also study the continuous dependence of these solutions with respect to the fractional parameter $s\nearrow 1$, i.e., we prove that the weak solutions for the fractional problem, obtained through the Rothe method, converge to a weak solution of the classical problem with $s=1$. 

Considering the special case of the hyperbolic obstacle problem, namely when $u\geq0$, by showing that the weak solutions of the viscous fractional problem are also very weak solutions in the sense introduced by \cite{bonafini2019Variational}, we show that, as $\nu\searrow 0$, they converge to very weak solutions to the fractional inviscid hyperbolic obstacle problem corresponding to \eqref{eq:problem} with $\nu=0$. These very weak fractional solutions, when $s\nearrow 1$, also converge to the very weak solutions of the classical vibrating membrane with obstacle. Our results extend the approach of \cite{bonafini2019Variational} to fractional non-homogeneous operators and show that their notion of very weak solution is stable as $s\nearrow 1$.

The paper is organized in the following way: in Section \ref{sec:functional_setting} we introduce the notions of weak solutions to \eqref{eq:problem} and the notion of very weak solutions for the obstacle problem, corresponding to $\mathrm{Dom(\beta)}=[0,+\infty)$. We also discuss the functional setting, namely we recall the notion of the Riesz fractional gradient $D^s$ that was introduced in \cite{shieh2015On}, as well as its connections with the fractional Sobolev space $H^s_0(\Omega)$; in Section \ref{sec:approximating_problem} we study a family of penalized problems which, together with the estimates obtained in Section \ref{sec:apriori_estimates}, yields approximated solutions that are used in the proof of the existence of a weak solution to \eqref{eq:problem} in Section \ref{sec:existence_s_fixed}; for the obstacle problem, we obtain in Section \ref{sec:from_viscous_to_inviscid} the existence of very weak solutions with weak solutions of the viscous problem by letting $\nu\searrow 0$; finally, in Section \ref{sec:existence_s_to_1} we study the stability of the weak solutions of the general problem obtained in Section \ref{sec:existence_s_fixed} and the stability of the very weak solutions to the obstacle problem obtained in Section \ref{sec:from_viscous_to_inviscid}, as $s\nearrow 1$.

For simplicity of presentation, in \eqref{eq:problem} when $\nu>0$ we have chosen the same exponent $0<s\leq 1$ both in the elastic and in the viscoelastic terms. Nevertheless all the results of this work can be generalized to the case where the fractional parameter $s$ in the viscous terms is replaced by a possible independent exponent $r$, with $0<r\leq 1$. These changes would require to work with the fractional Sobolev spaces $H^s_0(\Omega)$ for the $u$ and $H^r_0(\Omega)$ for the $\dot u$, with similar convergences of those solutions with respect to the fractional parameters $s\nearrow 1$ and $r\nearrow 1$.

\section{The functional setting for weak and very weak solutions}\label{sec:functional_setting}

\subsection{Fractional Sobolev spaces and operators of fractional divergence form}

In this subsection we recall some results concerning the Riesz fractional gradient $D^s$, with $s\in (0,1)$, as defined in \cite{shieh2015On}, focusing on its connection with the fractional Sobolev space $H^s(\R^d)$. For smooth functions with compact support $\varphi\in C^\infty_c(\R^d)$, we define $D^s\varphi$ as
\begin{equation}
    D^s \varphi = D(I_{1-s} \varphi)
\end{equation}
where $I_{1-s}:=(-\Delta)^{s-1}$ is the Riesz potential of order $1-s$ (see \cite{grafakos2014Classical} for more information about this family of potentials).

One of the major advantages of this definition of fractional gradient is that it satisfies the duality property
\begin{equation}
    \int_{\R^d}{D^s\varphi\cdot \Phi}\,dx=-\int_{\R^d}{\varphi D^s\cdot\Phi}\,dx
\end{equation}
for functions $\Phi\in C^\infty_c(\R^d;\R^d)$ and $\varphi\in C^\infty_c(\R^d)$.
And so, as a consequence of this property and $D^s\cdot\Phi\in C^\infty(\R^d)$, we can weaken the concept of Riesz fractional gradient for functions $u\in L^1(\R^d)+L^{\infty}(\R^d)$ as the function $D^su := g\in L^1_{\mathrm{loc}}(\R^d)$ for which we have
\begin{equation}
    \int_{\R^d}{g\cdot \Phi}\,dx=-\int_{\R^d}{u D^s\cdot\Phi}\,dx, \quad \forall \Phi\in C^\infty_c(\R^s;\R^d).
\end{equation}

With this notion of weak fractional gradient, we define the fractional Sobolev space $H^s(\R^d)$, similarly to what is done in the classical setting, as
 \begin{equation}
     H^s(\R^d):=\{u\in L^2(\R^d):\, D^s u\in L^2(\R^d;\R^d)\},
\end{equation}
with the Hilbertian norm
\begin{equation}
    \|u\|_{H^s(\R^d)}=\left(\|u\|^2_{L^2(\R^d)}+\|D^s u\|^2_{L^2(\R^d;\R^d)}\right)^{1/2}.
\end{equation}
As it was observed in \cite{shieh2015On}, this space coincides with the usual fractional Sobolev space defined in terms of the Gagliardo seminorm, see \cite{di2012Hitchhikers} for more information about these spaces.

Since in \eqref{eq:problem} we are imposing homogeneous Dirichlet exterior conditions on the function $u$, we are not going to use explicitly the space $H^s(\R^d)$. So, we introduce the space
\begin{equation}
    H^s_0(\Omega)=\overline{C^\infty_c(\Omega)}^{\|\cdot\|_{H^s(\R^d)}},
\end{equation}
that consists of all functions in $H^s(\R^d)$ that can be approximated, with respect to the norm of this space, by smooth functions with compact support in $\Omega$. It is important to notice that when $\Omega$ has Lipschitz boundary, which we shall assume, this space can be identified with the space of all functions in $H^s(\R^d)$ that equal to $0$ a.e. in $\R^d\setminus\Omega$. In this work we also assume that $\Omega$ is an open bounded set.

With this characterization of the fractional Sobolev spaces $H^s_0(\Omega)$ in terms of the fractional gradient we can extend directly the classical variational methods to the fractional framework. As shown in \cite{shieh2015On}, some of the properties that can be generalized to the fractional case are the fractional Sobolev inequalities. Another useful inequality is the following fractional Poincaré inequality:
\begin{lemma}\label{lemma:poincares_inequality}
    Let $s\in(0,1)$ and $u\in H^s_0(\Omega)$. Then, there exists a positive constant $C>0$ independent of $s$ such that
    \begin{equation}
        \|u\|_{L^2(\Omega)}\leq \frac{C}{s}\|D^s u\|_{L^2(\R^d;\R^d)}.
    \end{equation}
    Consequently, the space $H^s_0(\Omega)$ can be endowed with the norm $u\mapsto \|D^s u\|_{L^2(\R^d;\R^d)}$.
\end{lemma}
\begin{proof}
    Check \cite[Theorem~2.9]{bellido2021Gamma}.
\end{proof}
This property allow us to study the Dirichlet problem for some classes of fractional partial differential equations. In particular, the fractional Laplacian $(-\Delta)^s$, for sufficiently regular functions $\varphi$, at a point $x\in\R^d$ can be written as
\begin{equation}
    -D^s\cdot D^s \varphi(x)=(-\Delta)^s \varphi(x)= c_{d,s}\int_{\R^d}{\frac{\varphi(x)-\varphi(y)}{|x-y|^{d+2s}}}\,dx,
\end{equation}
where $c_{d,s}$ is a normalizing constant, allowing the fractional Laplacian to be written in the fractional divergence form. In fact, this approach allows us to consider many other fractional differential operators, linear or non-linear. To study \eqref{eq:problem}, we restrict ourselves to the linear anisotropic operators in fractional divergence form
\begin{equation}
    \mathscr{A}^s:H^s_0(\Omega)\to H^{-s}(\Omega)\qquad  \mbox{ and } \quad \mathscr{B}^s:H^s_0(\Omega)\to H^{-s}(\Omega),
\end{equation}
which are defined in the sense of distributions as
\begin{equation}\label{eq:definition_of_operators}
    \mathscr{A}^s \varphi=-D^s\cdot(AD^s \varphi)\quad \mbox{ and }\quad  \mathscr{B}^s \varphi=-D^s\cdot(BD^s \varphi)
\end{equation}
for matrix-valued functions $A, B:\R^d\to\R^{d\times d}$ that are assumed to be bounded, measurable, strictly elliptic and satisfying, for all $\eta, \zeta \in \R^d$, 
\begin{equation}
    a_*|\eta|^2\leq A(x)\eta\cdot\eta,\quad A(x)\eta\cdot\zeta\leq a^*|\eta||\zeta|\qquad \mbox{ and }\qquad
    b_*|\eta|^2\leq B(x)\eta\cdot\eta,\quad B(x)\eta\cdot\zeta\leq b^*|\eta||\zeta|,
\end{equation}
and also $B$ to be symmetric.

Another interesting property of the Riesz fractional gradient is its continuous behaviour as $s\nearrow 1$. Although $D^s$ is a nonlocal operator and  $D=D^1$ is local, it can be easily proved, using methods from Fourier analysis (see for example \cite[Lemma~3.7]{lo2023Class}), that
\begin{equation}
    D^s \varphi\to D \varphi \quad \mbox{ in }L^2(\R^d;\R^d) \quad \mbox { as } s\nearrow 1,
\end{equation}
provided that $\varphi\in H^1(\R^d)$. This idea can then be generalized for functions depending on time:
\begin{lemma}\label{lemma:strong_convergence_Ds_D_fixed_function}
    Let $T>0$, $s\nearrow 1$ and $\varphi\in L^2(0,T; H^s_0(\Omega))$. Then $D^s \varphi\to D\varphi$ in $L^2(0,T;L^2(\Omega))$.
\end{lemma}
\begin{proof}
    From Plancherel's theorem and from the fact that  $(|2\pi\xi|^{s-1}-1)^2\leq (|2\pi\xi|^{-1}-1)^2$, we have
    \begin{align*}
        \left(\int_0^T{\int_{\R^d}{4\pi^2|\xi|^2(|2\pi\xi|^{s-1}-1)^2|\hat{\varphi}(t,\xi)|^2}\,d\xi}\,dt\right)^{1/2}
        &\leq \left(\int_0^T{\int_{\R^d}{4\pi^2|\xi|^2(|2\pi\xi|^{-1}-1)^2|\hat{\varphi}(t,\xi)|^2}\,d\xi}\,dt\right)^{1/2}\\
        &=\left(\int_0^T{\int_{\R^d}{(\varphi-D\varphi)^2}\,dx}\,dt\right)^{1/2}\\
        &\leq \|\varphi\|_{L^2(0,T;L^2(\Omega))}+\|D\varphi\|_{L^2(0,T;L^2(\Omega))}.
    \end{align*}
    The result then follows by applying the dominated convergence theorem as well as Plancherel's theorem.
\end{proof}

In order to study the stability properties of the solutions of \eqref{eq:problem}, namely their convergence when $s\nearrow 1$, we also need a result that connects the fractional gradient of these solutions to the classical gradient of the limit of these functions:
\begin{lemma}\label{lemma:bellido_compactness}
    Let us consider a sequence  such that $s\nearrow 1$ and let $u_s\in H^s_0(\Omega)$ for each $s$, such that $\|D^s u_s\|_{L^2(\R^d;\R^d)}$ is uniformly bounded with respect to $s$, $\sigma<s<1$. Then, there exists a function $u\in H^1_0(\Omega)$ and a subsequence $\{s_n\}_{n\in\N}$ such that
    \begin{equation}
        u_{s_n}\to u \mbox{ in } H^\sigma_0(\Omega) \quad \mbox{ and } \quad D^{s_n}u_{s_n}\rightharpoonup Du \mbox{ in } L^2(\R^d;\R^d),
    \end{equation}
    as $n\to\infty$.
\end{lemma}
\begin{proof}
    See \cite[Theorem~4.2]{bellido2021Gamma} and \cite[Proposition~2.4.11]{campos2021Lions}.
\end{proof}

\subsection{Obstacles, subdifferentials and bounded variation}

Physically, when the membrane touches the obstacle, the latter is going to exert a force on the former causing it to suddenly change the direction of its motion in that point and making the velocity be discontinuous. To represent that force we use a maximal monotone graph $\beta$ in $\R\times\R$. We shall assume that $0\in\beta(0)$ and that the domain of $\beta$ can be any sub-interval of $\R$, the former being a compatibility hypothesis with the exterior conditions and the latter describing the nature of the problem. In fact, if there is no obstacle, then $\beta$ is a monotone non-decreasing function, possibly with a finite number of discontinuities, which has $\mathrm{Dom(\beta)}=\R$. In the presence of an obstacle, $u\geq a$ or $u\leq b$, or of two obstacles $a\leq u\leq b$, for $a,b\in \R$ and $a\leq0\leq b$, we have $\mathrm{Dom(\beta)}=[a,+\infty)$, $\mathrm{Dom(\beta)}=(-\infty, b]$ and $\mathrm{Dom(\beta)}=[a,b]$, respectively.

Recalling the theory of maximal monotone operators, see \cite[Chapter~2]{barbu2010Nonlinear},  $\beta$ in $\R\times\R$ can be written as the subdifferential of a lower semicontinuous, convex, proper function $j$ on $\R$, i.e., for each $x\in\R$
\begin{equation}
    \xi\in\beta(x)=\partial j(x) \Leftrightarrow j(x)-j(y)\leq \xi\cdot (x-y)\quad \mbox{ for all } y\in\R,
\end{equation}
with $\mathrm{Dom}(j)=\mathrm{Dom}(\beta)$ and $j(0)=\min{j}=0$. As an example, for the obstacle problems, we denote the indicator function of the interval $[a,b]$ by $j=I_{[a,b]}$ and its subdifferential by $\beta=\partial I_{[a,b]}$, using the abuse of notation $b=+\infty$ and $a=-\infty$ for the lower and upper obstacle problems, respectively.

We indicate also by $\beta$ the maximal monotone operator in the space  $L^2(0,t; L^2(\Omega))$ for $t \in(0,T]$ corresponding to the subdifferential of the functional
\begin{equation}
    \begin{aligned}
        \mathcal{J}_t: L^2(0,t; L^2(\Omega))&\to [0,+\infty]\\
        u&\mapsto \int_0^t{\int_\Omega{j(u)}\,dx}\,d\tau.
    \end{aligned}
\end{equation}
Formally we can now write equation \eqref{eq:problem} with a function $\xi\in \beta(u)$ such that it becomes
\begin{equation}\label{eq:equation_to_define_spaces}
     \ddot{u}-D^s\cdot(A D^s u)-\nu D^s\cdot (B D^s \dot u) + \xi= g.
\end{equation}
By testing this equation with a sufficiently regular function $\varphi$ and then using integration by parts in time, provided that $\xi\in L^2(0,t;L^2(\Omega))$, we get
\begin{equation}
    \begin{aligned}
        \int_{\Omega}{\dot{u}(t) \varphi(t)}\,dx
        -\int_{\Omega}{\dot{u}(0) \varphi(0)}\,dx
        &-\int_0^t{\int_\Omega{\dot{u} \dot{\varphi}}\,dx}\,d\tau+\int_0^t{\int_\Omega{\xi\varphi}\,dx}\,d\tau\\
        &\quad +\int_0^t{\int_{\R^d}{\left(A D^s u\cdot D^s \varphi + \nu B D^s \dot{u} \cdot D^s \varphi\right)}\,dx}\,d\tau
        =\int_0^t{\int_{\Omega}{g\varphi}\,dx}\,d\tau.
    \end{aligned}
\end{equation}
However, we are not able to prove the regularity $\xi\in L^2(0,t;L^2(\Omega))$, and therefore this condition must be weaken. Following \cite{bonetti2017On}, we introduce the functional spaces
\begin{equation}
    \mathcal{V}_{s,t}:= H^1(0,t; L^2(\Omega))\cap L^2(0,t; H^s_0(\Omega))
\end{equation}
for $0<s\leq1$ and $0<t<T$. For simplicity, when $t=T$, we simply write
\begin{equation}
    \mathcal{V}_s:=\mathcal{V}_{s,T}=H^1(0,T; L^2(\Omega))\cap L^2(0,T; H^s_0(\Omega)).
\end{equation}
Indeed, for the analysis in this article, we need to deal with $\xi$ being an element of $\mathcal{V}'_{s,t}$, the dual of $\mathcal{V}_{s,t}$. So, when we restrict $\mathcal{J}_t$ to functions on $\mathcal{V}_{s,t}$, this new functional is convex and lower semicontinuous on $\mathcal{V}_{s,t}$ and hence, we can weaken the constraint $\beta$ to a new maximal monotone graph $\beta_{s,t}$ defined as
\begin{equation}
    \xi \in \beta_{s,t}(u)\Leftrightarrow u\in \mathcal{V}_{s,t},\,\,\,\xi\in \mathcal{V}'_{s,t} \,\,\,\mbox{ and }\,\,\, \mathcal{J}_t|_{\mathcal{V}_{s,t}}(u)-\mathcal{J}_t|_{\mathcal{V}_{s,t}}(v)\leq\langle \xi, u-v\rangle_{\mathcal{V}'_{s,t}\times \mathcal{V}_{s,t}}, \quad \mbox{ for all } v\in \mathcal{V}_{s,t}.
\end{equation}
Similarly, we also define $\mathcal{J}:=\mathcal{J}_{T}$ and $\beta_s:=\beta_{s,t}$. Moreover, this operator also has good stability properties, useful when for the proof of Theorem \ref{thm:existence_weak_solutions}:
\begin{lemma}
    Let $\beta_n$ be a sequence of monotone operators that converges in the sense of graphs, in $\mathcal{V}'_{s,t}\times\mathcal{V}_{s,t}$, to $\beta_{s,t}$. Let us also assume that we have a sequence of functions $v_n\rightharpoonup v$ in $\mathcal{V}_{s,t}$ and a sequence of distributions $\xi_n\rightharpoonup \xi$ in $\mathcal{V}'_{s,t}$, such that $\xi_n=\beta_n(v_n)$ for each $n\in\N$. Then, if in addition we have that
    \begin{equation}
        \limsup{\langle \eta_n, v_n\rangle_{\mathcal{V}'_{s,t}\times\mathcal{V}_{s,t}}}\leq \langle \eta, v\rangle_{\mathcal{V}'_{s,t}\times\mathcal{V}_{s,t}},
    \end{equation}
    we can conclude that $\xi\in \beta_{s,t}(v)$.
\end{lemma}
\begin{proof}
    This is an easy consequence of \cite[Proposition~2.59]{attouch1984Variational}. See the argument in Step 4 of the proof of Theorem 2.5 of \cite{bonetti2017On}.
\end{proof}

When dealing with these spaces $\mathcal{V}'_{s,t}$ and $\mathcal{V}'_{s}$ we need to be able to relate them. For that, following \cite{bonetti2017On}, let us introduce also the space
\begin{equation}
    \overline{\mathcal{V}}_{s,t}=\left\{v\in\mathcal{V}_{s,t}:\, v\equiv 0 \mbox{ on } \{t\}\times\Omega \mbox{ in the sense of traces}\right\}
\end{equation}
which is a closed subspace of $\mathcal{V}_{s,t}$. Then, for a distribution $\eta\in\mathcal{V}'_s$, we define its restriction to $(0,t)$ by setting
\begin{equation}\label{eq:restriction_distribution}
    \langle\eta|_{(0,t)},\varphi\rangle_{\mathcal{V}'_{s,t}\times\mathcal{V}_{s,t}}:=\langle\eta,\Tilde{\varphi}\rangle_{\mathcal{V}'_{s}\times\mathcal{V}_{s}}
\end{equation}
for $\varphi\in \overline{\mathcal{V}}_{s,t}$ and $\Tilde{\varphi}$ being its extension by $0$ for time larger than $t$.

\vspace{3mm}

Returning to the observation that we made in the beginning of this subsection about the behavior of the membrane when it hits the obstacle, we remark that the sudden change of direction of motion is directly related to an instantaneous abrupt change of velocity. This means that we should not expect the velocity to be continuous, but only to have left and right derivatives. Hence, we recall some properties concerning the space of function of bounded variation. Let $X$ be a Banach space and
\begin{equation}
    BV(0,T;X)= \left\{f:[0,T]\to X:\,\mathrm{var}_0^T(f):=\sup{\sum_{j=1}^n{\|f(t_j)-f(t_{j-1})\|_X}}<\infty\right\}
\end{equation}
where the supremum is taken over all finite partitions $0\leq t_0<t_1<...<t_n\leq T$. 
The property of these spaces regarding the left and right limits can be summarized in the following lemma:

\begin{lemma}
    Let $X$ be a Banach space and $f\in BV(0,T; X)$. Then for every $t\in[0,T)$ there exists a function in $X$ called the right-limit of $f$ at $t$, which we denote by $f(t^+)$, satisfying
    \begin{equation}
        \lim_{\tau\searrow t}{f(\tau)}\to f(t^+) \mbox{ in } X.
    \end{equation}
    Similar, there exists a left-limit if $f$ for every $t\in(0,T]$, denoted $f(t^-)$, for which
    \begin{equation}
        \lim_{\tau\nearrow t}{f(\tau)}\to f(t^-) \mbox{ in } X.
    \end{equation}
\end{lemma}
\begin{proof}
    Check \cite[Proposition~4.2]{moreau1988Bounded}
\end{proof}

In addition, the space of functions with bounded variation also has other properties that are important for our analysis. One of such results is the generalization of Helly's selection lemma:
\begin{lemma}\label{lemma:generalized_helly_selection}
    Let $X$ be a separable, reflexive Banach space with separable dual and $\{u_j\}\subset BV(0,T;X)$ uniformly with respect to $j$. Then, there exists a subsequence, still denoted $u_j$, and a function $u\in BV(0,T;X)$ such that $u_j(t)\overset{\ast}{\rightharpoonup} u(t)$ in $X$ for every $t\in[0,T]$.
\end{lemma}
\begin{proof}
    Check \cite[Lemma~7.2]{dalMaso2006Quasistatic}.
\end{proof}

And finally, we present an adaptation of \cite[Proposition~11]{bonafini2019Variational} that is going to be extensively used throughout this work.

\begin{lemma}\label{lemma:limit_bv_functions_and_right_limit}
    Let $X$ be a separable, reflexive Banach spaces with separable dual and $\varphi\in L^2(\Omega)\subset X$. Consider also a sequence of functions $\{v_j\}_{j\in\N}\subset L^2(0,T;L^2(\Omega))$ such that $\|v_j(t)\|_{L^2(\Omega)}\leq C$ for all $t\in[0,T]$ and $j\in\N$, and $v_j\overset{\ast}{\rightharpoonup} v$ in $L^\infty(0,T; L^2(\Omega))$. If the sequence of functions $F^j:[0,T]\to \R$ defined as
    \begin{equation}
        F^j(t)=\int_{\Omega}{v_j(t)\varphi}\,dx
    \end{equation}
    is uniformly bounded in $BV(0,T)$ with respect to $j$, then the function $F:[0,T]\to \R$ defined by
    \begin{equation}
        F(t)=\int_{\Omega}{v(t)\varphi}\,dx
    \end{equation}
    also belongs to $BV(0,T)$. Moreover, if $\{v_j\}$ is also uniformly bounded in $BV(0,T; X)$, then $v$ has right limit at $0$ and $v(0^+)\in L^2(\Omega)$, which has the property that
    \begin{equation}
        \lim_{t\searrow 0}F(t)=\int_{\Omega}{v(0^+)\varphi}\,dx.
    \end{equation}
\end{lemma}
\begin{proof}
    Using the fact that $F_k$ are uniformly bounded in $BV(0,T)$ and hence, by Helly's selection theorem, there exists a function $\Bar{F}\in BV(0,T)$ such that $F_k(t)\to F(t)$ for every $t\in[0,T]$. Now we have to check that $\Bar{F}=F$. We start by taking $\varphi(t,x)=\eta(t)\psi(x)$ for some $\eta\in C^\infty_c(0,T)$. Since $v_k\overset{\ast}{\rightharpoonup} v$ in $L^\infty(0,T; L^2(\Omega))$, then by the dominated convergence theorem we have
    \begin{equation}
        \begin{aligned}
            \int_0^T{\int_\Omega{v(t)\psi}\,dx\,\eta(t)}\,dt
            &=\int_0^T{\int_\Omega{v(t)\varphi}\,dx}\,dt
            =\lim_{k\to\infty}{\int_0^T{\int_\Omega{v_k(t)\varphi}\,dx}\,dt}\\
            &=\lim_{k\to\infty}{\int_0^T{\int_\Omega{v_k(t)\psi}\,dx\,\eta(t)}\,dt}
            =\int_0^T{\lim_{k\to\infty}{\int_\Omega{v_k(t)\psi}\,dx}\,\eta(t)}\,dt
            =\int_0^T{\Bar{F}(t)\eta(t)}\,dt
        \end{aligned}
    \end{equation}
    Consequently, we have
    \begin{equation}
        \int_0^T{\left(\int_\Omega{v(t)\psi}\,dx-\Bar{F}(t)\right)\eta(t)}\,dt,
    \end{equation}
    which, due to the arbitrariness of $\eta$, allow us to say that $F=\Bar{F}$ for a.e. $t\in[0,T]$.

    Now let $\{v_j\}$ be also uniformly bounded in $BV(0,T; X)$. By Lemma \ref{lemma:generalized_helly_selection}, then we can infer that $v\in BV(0,T;X)$ and we can extract a subsequence such that $v_j(t)\overset{\ast}{\rightharpoonup} v(t)$ in $X$. Consequently, since $v\in BV(0,T; X)$, it makes sense to define $v(0^+)\in X$. At the same time, since $F(t)$ is in $BV(0,T)$, then it also makes sense $\lim_{t\searrow 0}{F(t)}=F(0^+)\in\R$. Moreover, using the fact that \begin{equation}
        \|v(t)\|_X\leq \liminf_{j\to\infty}{\|v_j(t)\|_{X}}\leq C\liminf_{j\to\infty}{\|v_j(t)\|_{L^2(\Omega)}}\leq C',\quad \forall t\in[0,T]
    \end{equation}
    we deduce that $v(t)\rightharpoonup w$ in $L^2(\Omega)$ as $t\searrow 0$. Notice that this limit is independent of the choice of the sequence $t_n\to 0$ because $F$ is $BV(0,T)$.
    
    We now want to check that $w=v(0^+)$ in $X$. For that consider an arbitrary but fixed function $\psi\in X^*\subset L^2(\Omega)$ and observe that
    \begin{equation*}
        \langle \psi, w\rangle_{X^*\times X}=\int_{\Omega}{w\psi}\,dx=\lim_{t\searrow 0}{\int_{\Omega}{v(t)\psi}\,dx}=\lim_{t\searrow 0}{\langle v(t),\psi\rangle_{X\times X^*}}=\langle\psi, v(0^+)\rangle_{X^*\times X}.
    \end{equation*}
\end{proof}

\subsection{Weak and very weak solutions}
Having fixed the functional framework,  we define the notion of weak solution to the viscous problem, $\nu>0$, similarly to the one used in \cite{bonetti2017On} to study the two obstacle problem in the special local isotropic case, i.e., $s=1$ and $A=B=Id$. 

\begin{definition}[Weak solution]\label{def:weak_solution}
    Let $\nu>0$, $s\in(0,1]$ and $w_0,w_1\in H^s_0(\Omega)$. We say that the pair $(u,\xi)$ is a weak solution of \eqref{eq:problem} if:
    \begin{itemize}
        \item[a)] $u\in H^1(0,T; H^s_0(\Omega))\cap W^{1,\infty}(0,T; L^2(\Omega))$ with $u(0)=w_0$, $\dot{u}(0)=w_1$ and $\xi\in\beta_s(u)\subset \mathcal{V}'_{s}$;
        \item[b)] the following identity holds
        \begin{equation}\label{eq:definition_weak_solution_for_T}
            \begin{aligned}
                \int_{\Omega}{\dot{u}(T) \varphi(T)}\,dx
                -\int_{\Omega}{w_1 \varphi(0)}\,dx
                &-\int_0^T{\int_\Omega{\dot{u} \dot{\varphi}}\,dx}\,d\tau
                +\langle \xi, \varphi\rangle_{\mathcal{V}'_{s}\times\mathcal{V}_{s}}\\
                &\quad +\int_0^T{\int_{\R^d}{\left(A D^s u\cdot D^s \varphi + \nu B D^s \dot{u} \cdot D^s \varphi\right)}\,dx}\,d\tau
                =\int_0^T{\int_{\Omega}{g\varphi}\,dx}\,d\tau,
            \end{aligned}
        \end{equation}
        for all $\varphi\in\mathcal{V}_{s}$.
        \item[c)] there exists for every $t\in(0,T]$ distribution $\xi_t\in\beta_{s,t}(u)\subset\mathcal{V}'_{s,t}$ corresponding to the restriction of the distribution $\xi$ in the sense of \eqref{eq:restriction_distribution}, such that
        \begin{equation}\label{eq:definition_weak_solution}
            \begin{aligned}
                \int_{\Omega}{\dot{u}(t) \varphi(t)}\,dx
                -\int_{\Omega}{w_1 \varphi(0)}\,dx
                &-\int_0^t{\int_\Omega{\dot{u} \dot{\varphi}}\,dx}\,d\tau
                +\langle \xi_t, \varphi\rangle_{\mathcal{V}'_{s,t}\times\mathcal{V}_{s,t}}\\
                &\quad +\int_0^t{\int_{\R^d}{\left(A D^s u\cdot D^s \varphi + \nu B D^s \dot{u} \cdot D^s \varphi\right)}\,dx}\,d\tau
                =\int_0^t{\int_{\Omega}{g\varphi}\,dx}\,d\tau,
            \end{aligned}
        \end{equation}
       for all $\varphi\in\mathcal{V}_{s,t}$.
    \end{itemize}
\end{definition} 
\begin{remark}
    In the classical case $s=1$, all the integrals in $\R^d$ are in fact integrals in $\Omega$ as $Du=0$ in $\R^d\setminus\Omega$.
\end{remark}
In the particular case of the lower or upper obstacle problem or in the two obstacles problem, i.e., when $\beta=\partial I_{[a,b]}$ with $\mathrm{Dom}(\beta)=[a,b]\ni 0$ being an halfline or a finite interval of $\R$, respectively, this definition yields a variational inequality.  By introducing the following non-empty convex set
\begin{equation}
    K^s=\{\psi\in H^1(0,T; H^s_0(\Omega)):\, v\in \mathrm{Dom}(\beta) \mbox{ a.e. in } Q_T\},
\end{equation}
we can immediately see that for a weak solution $u$ and a distribution $\xi\in\beta_s(u)$,
\begin{equation}
    \langle \xi, \psi-u\rangle_{\mathcal{V}'_s\times\mathcal{V}_s}\leq 0, \quad \mbox{ for all } \psi\in K^s.
\end{equation}
Therefore, we have $u\in K^s$ and, by testing \eqref{eq:definition_weak_solution} with $\varphi=\psi-u$, where $\psi$ is an arbitrary function in $K^s$, and using the previous inequality, we deduce that
\begin{equation}\label{eq:variational_inequality}
    \begin{aligned}
        &\int_\Omega{\dot{u}(T) (\psi(T)-u(T))}\,dx
        -\int_\Omega{w_1 (\psi(0)-w_0)}\,dx
        -\int_0^T{\int_\Omega{\dot{u} (\dot{\psi}-\dot{u})}\,dx}\,dt\tau\\
        &\qquad+\int_0^T{\int_{\R^d}{\big(A D^s u\cdot D^s(\psi-u) + \nu B D^s \dot{u} \cdot D^s(\psi-u)\big)}\,dx}\,d\tau
        \geq\int_0^T{\int_{\Omega}{g(\psi-u)}\,dx}\,d\tau, \quad \forall \psi\in K^s
    \end{aligned}
\end{equation}
\begin{remark}
    Notice that, by density, the previous inequality also holds for all 
    \begin{equation}
        \psi\in\Tilde{K}^s=\{\psi\in \mathcal{V}_s:\, v\in \mathrm{Dom}(\beta) \mbox{ a.e. in } Q_T\}
    \end{equation}
    instead of $K^s$.
\end{remark}
This variational inequality approach \eqref{eq:variational_inequality} was explored in \cite{jarusek1992Variational} to study the existence of a solution to the viscous hyperbolic lower obstacle problem $u\geq0$,  within the local $s=1$ isotropic setting and with Neumann boundary condition (although the authors make also a remark about the existence of solution with Dirichlet boundary condition).

It is important to emphasize that when the viscous term vanishes, $\nu=0$, we don't know whether or not there is a weak solution to the obstacle problem. In fact, determining whether the variational inequality \eqref{eq:variational_inequality} or the weak formulation of Definition \ref{def:weak_solution} associated with the wave equation have solutions is still an open problem. The issue is that, without the \textit{a priori} estimate that one obtains from the viscosity term, one cannot get the strong convergence of the velocity $\dot{u}$ in $L^2(0,T; L^2(\Omega))$ as $\nu\to 0$ needed to deal with the space-time integral that arises from the integration by parts of the acceleration. 

Nevertheless, when dealing with the lower obstacle problem $u\geq 0$, there is a weaker notion of solution to the obstacle problem without the viscosity term recently introduced in \cite{bonafini2019Variational}, which is valid also for obstacles depending on the space variable. This weaker notion of solution can also be extended to the viscoelastic obstacle problem, which will be called here a very weak solution.

\begin{definition}[Very weak solution for the viscous problem]\label{def:vey_weak_solution}
    Let $\nu> 0$, $s\in(0,1]$ and $w_0\in H^s_0(\Omega)$, $w_1\in L^2(\Omega)$. We say that $u=u(t,x)\geq 0$ is a very weak solution of \eqref{eq:problem} in $(0,T)$ if $u\in H^1(0,T; H^s_0(\Omega))$, $\dot{u}\in BV(0,T; X)$, $\dot{u}(0^+)\in L^2(\Omega)$, $u(t,x)\geq 0$ for a.e. $(t,x)\in Q_T$ and it satisfies the inequality
    \begin{equation}
        -\int_0^T{\int_\Omega{\dot{u}\dot\varphi}\,dx}\,dt+\int_0^T{\int_{\R^d}{AD^su\cdot D^s \varphi}\,dx}\,dt+\nu\int_0^T{\int_{\R^d}{ BD^s\dot{u}\cdot D^s \varphi}\,dx}\,dt\geq \int_\Omega{w_1\varphi(0)}\,dx+\int_0^T{\int_{\Omega}{g\varphi}\,dx}\,d\tau
    \end{equation}
    for all $\varphi\in\left\{\varphi\in\mathcal{V}_s: \varphi\geq 0\mbox{ and } \mathrm{supp}(\varphi)\subset[0,T)\right\}$, as well as the initial conditions
    \begin{equation}\label{eq:initial_conditions_very_weak_solutions}
        u(0,\cdot)=u_0 , \qquad \int_\Omega{(\dot{u}(0^+)-w_1)(\psi-w_0)}\,dx\geq 0\quad \forall\psi\in H^s_0(\Omega), \psi\geq 0.
    \end{equation}
\end{definition}

This definition is based on the variational inequality \eqref{eq:variational_inequality} but replaces $\psi-u$ by $\varphi$, replacing the problematic term $\dot{u}(\dot{\psi}-\dot{u})$ by $\dot{u}\dot{\varphi}$. As for the inequality \eqref{eq:initial_conditions_very_weak_solutions}, it represents what happens immediately after the membrane starts moving. In fact, it says that if the membrane if far from the obstacle, then $\dot{u}(0^+)=w_1$, but if the initial position of the membrane is already in contact with the obstacle, the velocity $\dot{u}(0^+)$ might differ from the initial velocity. The fact that this information comes in the form of an inequality is due to the fact that we do not specify if the collision between the membrane and the obstacle is elastic, perfectly inelastic or something in between.

As the hyperbolic elastic obstacle problem requires less regularity of the solution, we introduce the notion of very weak solution similarly to \cite{bonafini2019Variational}.

\begin{definition}[Very weak solution for the inviscid problem]\label{def:vey_weak_solution_inviscid}
    Let $\nu=0$, $s\in(0,1]$ and $w_0\in H^s_0(\Omega)$, $w_1\in L^2(\Omega)$. We say that $u=u(t,x)\geq 0$ is a very weak solution of \eqref{eq:problem} in $(0,T)$ if $u\in H^1(0,T; L^2(\Omega))\cap L^2(0,T; H^s_0(\Omega))$, $\dot{u}\in BV(0,T; X)$, $\dot{u}(0^+)\in L^2(\Omega)$, $u(t,x)\geq 0$ for a.e. $(t,x)\in Q_T$ and it satisfies the inequality
    \begin{equation}
        -\int_0^T{\int_\Omega{\dot{u}\dot\varphi}\,dx}\,d\tau
        +\int_0^T{\int_{\R^d}{AD^su\cdot D^s \varphi}\,dx}\,d\tau\geq \int_\Omega{w_1\varphi(0)}\,dx+\int_0^T{\int_{\Omega}{g\varphi}\,dx}\,d\tau
    \end{equation}
    for all $\varphi\in\left\{\varphi\in\mathcal{V}_s: \varphi\geq 0\mbox{ and } \mathrm{supp}(\varphi)\subset[0,T)\right\}$, as well as the initial conditions
    \begin{equation}\label{eq:initial_conditions_very_weak_solutions_inviscid_problem}
        u(0,\cdot)=u_0 , \qquad \int_\Omega{(\dot{u}(0^+)-w_1)(\psi-w_0)}\,dx\geq 0\quad \forall\psi\in H^s_0(\Omega), \psi\geq 0.
    \end{equation}
\end{definition}

\section{The penalized viscous problem}\label{sec:approximating_problem}

In order to prove the existence of solutions, as it is standard with maximal monotone operators, we consider approximations to the problem \eqref{eq:problem} by a family of regularized problems, similarly to \cite{jarusek1992Variational, bonetti2017On}.

Let $j^\varepsilon$ be the Moreau-Yosida regularization of $j$, see \cite{barbu2010Nonlinear}, and define the function
\begin{equation}
    \beta^\varepsilon:=\partial j^\varepsilon=(j^\varepsilon)',
\end{equation}
which is a monotone and globally Lipschitz continuous function on the real line, with Lipschitz constant depending on $\varepsilon$.

We consider, for each $\varepsilon\in (0,1)$, the regularized problems
\begin{equation}\label{eq:regularized_problem_equation_form}
    \begin{cases}
        \ddot{u}^\varepsilon-D^s\cdot(A D^s u^\varepsilon)-D^s\cdot (\nu B D^s \dot{u}^\varepsilon) + \beta^\varepsilon(u^\varepsilon)= g & \mbox{ in } Q_t,\\
        u^\varepsilon=0 & \mbox{ on } (\R^d\setminus\Omega)\times(0,T),\\
        u^\varepsilon(0,\cdot)=w_0,\quad \dot{u}^\varepsilon(0,\cdot)=w_1 & \mbox{ on } \Omega,
    \end{cases}
\end{equation}
with solutions $u^\varepsilon$ satisfying the weak formulation
\begin{multline}\label{eq:regularized_problem}
    \int_\Omega{\dot{u}^\varepsilon(t)\varphi(t)}\,dx-\int_\Omega{w_1 \varphi(0)}\,dx-\int_0^t{\int_\Omega{\dot{u}^\varepsilon \   \dot{\varphi}}\,dx}\,d\tau+\int_0^t{\int_{\Omega}{\beta^\varepsilon(u^\varepsilon)\varphi}\,dx}\,d\tau\\
    +\int_0^t{\int_{\R^d}{\left(A D^s u^\varepsilon\cdot D^s \varphi + \nu B D^s \dot{u}^\varepsilon \cdot D^s \varphi\right)}\,dx}\,d\tau=\int_0^t{\int_{\Omega}{g\varphi}\,dx}\,d\tau,
\end{multline}
for every $\varphi\in \mathcal{V}_{s,t}$.

\begin{theorem}\label{thm:existence_approximated_solutions}
    Let $T > 0$, $w_0 \in H^s_0(\Omega), w_1 \in L^2(\Omega)$, and $g \in L^2(\Omega)$ be given. Then, for all $\varepsilon\in(0,1)$ there exists a function $u^\varepsilon$ with
    \begin{equation}
        u^\varepsilon\in L^\infty(0,T; H^s_0(\Omega)),
        \quad 
        \dot{u}^\varepsilon \in L^\infty(0,T; L^2(\Omega))\cap L^2(0,T; H^s_0(\Omega)),
        \quad 
        \ddot{u}^\varepsilon \in L^2(0,T; L^2(\Omega))
    \end{equation}
    and
    \begin{equation}
        u^\varepsilon(0,\cdot)=w_0
        \quad \mbox{ and } \quad 
        \dot{u}^\varepsilon(0,\cdot)=w_1.
    \end{equation} that solves \eqref{eq:regularized_problem} for every $t\in(0,T]$. Moreover, it also satisfies the energy identity
    \begin{equation}\label{eq:energy_equality_varepsilon}
        \mathcal{E}^\varepsilon(u^\varepsilon(t),\dot{u}^\varepsilon(t))+\int_0^t{\int_{\R^d}{\nu BD\dot{u}^\varepsilon\cdot\dot{u}^\varepsilon}\,dx}\,d\tau= \mathcal{E}^\varepsilon(w_0,w_1)+\int_0^t{\int_\Omega{g\dot{u}(\tau)}\,dx}\,d\tau,
    \end{equation}
    where
    \begin{equation}\label{eq:energy_functional}
        \mathcal{E}^\varepsilon(v,w)=\frac{1}{2}\left(\|w\|^2_{L^2(\Omega)}+\int_{\R^d}{A D^s v\cdot D^s v}\,dx\right)+\int_\Omega{j^\varepsilon(v)}\,dx.
    \end{equation}
\end{theorem}
\begin{proof}
    To prove the existence of functions $u^\varepsilon$ satisfying equation \eqref{eq:regularized_problem} we can apply the Rothe method, which consists of a discretization in time of the equation. For that, let us consider $n\in\N$, $h=T/n$ and $t_k=kh$ for $k=-1,0,...,n$. Then, the idea of this method is to approximate the hyperbolic equation \eqref{eq:regularized_problem_equation_form} by an elliptic equation written in the sense of distributions
    \begin{multline}\label{eq:time_discretized_equation_rothe_method}
        \int_\Omega{\frac{u^\varepsilon_{n, j}-2u^\varepsilon_{n, j-1}+u^\varepsilon_{n,j-2}}{h^2}\varphi}\,dx+\int_{\R^d}{AD^s u^\varepsilon_{n, j} D^s \varphi}\,dx\\
        +\nu\int_{\R^d}{\frac{BD^s u^\varepsilon_{n, j}-BD^s u^\varepsilon_{n, j-1}}{h} D^s \varphi}\,dx+\int_{\Omega}{\beta^\varepsilon(u^\varepsilon_{n,j}) \varphi}\,dx=\int_\Omega{g\varphi}\,dx
    \end{multline}
    for all $\varphi\in H^s_0(\Omega)$, with
    \begin{equation}\label{eq:rothe_constraints}
        u^\varepsilon_{n,0}=w_0,\quad u^\varepsilon_{n, -1}=w_0-hw_1.
    \end{equation}
    Notice that this problem is defined recursively in the sense that for each $j$, the unknown function in the equation \eqref{eq:time_discretized_equation_rothe_method} is the function $u_j$, while $u_{j-1}$ and $u_{j-2}$ are already known from the previous steps in the recursion or by the constraints \eqref{eq:rothe_constraints}.
    
    Through classical variational techniques for monotone operators, one can show that for each $n\in\N$ and each $j=1...,n$, there exists a unique function $u^\varepsilon_{n,j}$ that solves \eqref{eq:time_discretized_equation_rothe_method}.

    Since the problem that we are studying involves the first and the second derivative in time, it is useful for us to also consider discretized velocity and acceleration functions
    \begin{equation}
        v^\varepsilon_{n,j}=\frac{u^\varepsilon_{n,j}-u^\varepsilon_{n,j-1}}{h} \quad \mbox{ and }\quad a^\varepsilon_{n,j}=\frac{v^\varepsilon_{n,j}-v^\varepsilon_{n,j-1}}{h}.
    \end{equation}
    Moreover, in order to apply Rothe's method it also useful to consider the following piecewise constant functions
    \begin{equation}
        u^\varepsilon_n=\begin{cases}
            w_0, & t\in(-h, 0]\\
            u^\varepsilon_{n,j}, & t\in (t^n_{j-1}, t^n_j]
        \end{cases},
        \quad 
        v^\varepsilon_n=\begin{cases}
            v^\varepsilon_{n,0}, & t\in(-h, 0]\\
            v^\varepsilon_{n,j}, & t\in (t^n_{j-1}, t^n_j]
        \end{cases}
        \quad \mbox{ and } \quad
        a^\varepsilon_{n,j}=\begin{cases}
            a^\varepsilon_{n,0}, & t\in(-h, 0]\\
            a^\varepsilon_{n,j}, & t\in (t^n_{j-1}, t^n_j]
        \end{cases},
    \end{equation}
    as well as the following piecewise affine functions in $t\in [t^n_{j-1}, t^n_j]$
    \begin{equation}\label{eq:piecewise_affine_dscretization}
        U^\varepsilon_n(t)=u^\varepsilon_{n,j-1}+\frac{1}{h}(t-t^n_{j-1})(u^\varepsilon_{n,j}-u^\varepsilon_{n,j-1})\quad \mbox{ and }\quad V^\varepsilon_n(t)=v^\varepsilon_{n,j-1}+\frac{1}{h}(t-t^n_{j-1})(v^\varepsilon_{n,j}-v^\varepsilon_{n,j-1}).
    \end{equation}

    Having defined these functions, we first observe that if we test \eqref{eq:time_discretized_equation_rothe_method} with $\varphi=v^\varepsilon_{n,j}$, we obtain that for $n$ sufficiently small, there exists a positive constant $C_{\varepsilon,\sigma}>0$ independent of $n$, $j$ and $s$, but dependent of $\varepsilon$ and $\sigma$ such that
    \begin{equation}\label{eq:estimate_1}
        \|v^\varepsilon_{n,j}\|^2_{L^2(\Omega)}+\|D^s v^\varepsilon_{n,j}\|^2_{L^2(\R^d;\R^d)}+\|D^s u^\varepsilon_{n,j}\|_{L^2(\R^d;\R^d)}\leq C_{\varepsilon,\sigma}.
    \end{equation}
    At the same time, if instead we test with $w=a^\varepsilon_{n,j}$, we obtain that for $n$ sufficiently small, there exists a positive constant $C_{\varepsilon,\sigma}>0$ independent of $n$, $j$ and $s$, but dependent of $\varepsilon$ and $\sigma$ such that
    \begin{equation}\label{eq:estimate_2}
        \|a^n_j\|^2_{L^2(\Omega)}+\|D^s v ^n_j\|^2_{L^2(\R^d;\R^d)}\leq C_{\varepsilon,\sigma}.
    \end{equation}
    Moreover, by the properties of Hille-Yosida's theory, we know that
    \begin{equation}\label{eq:estimate_3}
        \|\beta^\varepsilon(u^\varepsilon_{n,j})\|_{L^2(\Omega)}\leq \frac{1}{\varepsilon}\|u^\varepsilon_{n,j}\|_{L^2(\Omega)}\leq \frac{C_{\varepsilon,\sigma}}{\varepsilon}
    \end{equation}

    From the way we have defined $u^\varepsilon_n, v^\varepsilon_n$ and $a^\varepsilon_n$, the estimates \eqref{eq:estimate_1}, \eqref{eq:estimate_2} and \eqref{eq:estimate_3} also hold for these functions (when replacing $u^\varepsilon_{n,j}, v^\varepsilon_{n,j}$ and $a^\varepsilon_{n,j}$ appropriately). As a consequence of these estimates
    \begin{align*}
        &u^\varepsilon_n\rightharpoonup u^\varepsilon \mbox{ in } L^2(0,T; H^s_0(\Omega)),\\
        &\beta^\varepsilon(u^\varepsilon_n)\rightharpoonup \xi^\varepsilon \mbox{ in } L^2(0,T; L^2(\Omega)),\\
        &v^\varepsilon_n\rightharpoonup v^\varepsilon \mbox{ in } L^2(0,T; H^s_0(\Omega)),\quad \mbox{ and }\\
        &a^\varepsilon_n\rightharpoonup a^\varepsilon \mbox{ in } L^2(0,T; L^2(\Omega)).
    \end{align*}
    At the same time, since $\dot{U}^\varepsilon_n(t)=v^\varepsilon_n$ and $\dot{V}^\varepsilon_n(t)=a^\varepsilon_n$ for each $t\in(t^n_{j-1}, t^n_n)$, and
    \begin{equation*}
        \|U^\varepsilon_n(t)-u^\varepsilon_n(t)\|_{L^2(\Omega)}+\|V^\varepsilon_n(t)-v^\varepsilon_n(t)\|_{L^2(\Omega)}\leq C_{\varepsilon,\sigma}h,
    \end{equation*}
    we deduce that $v^\varepsilon=\dot{u}^\varepsilon$ and $a^\varepsilon=\ddot{u}^\varepsilon$.
    Moreover, due to Aubin-Lions lemma and Poincaré's inequality, we can extract further a subsequence such that
    \begin{align*}
        &V^\varepsilon_n(t)\to \dot{u}^\varepsilon(t) \mbox{ in } L^2(\Omega) \mbox{ for all } t\in[0,T], \\
        &u^\varepsilon_n\to u^\varepsilon \mbox{ in } L^2(0,T; L^2(\Omega)), \qquad \mbox{ and }\\
        &v^\varepsilon_n\to \dot{u}^\varepsilon \mbox{ in } L^2(0,T; L^2(\Omega)).
    \end{align*}
    From the definition of the piecewise constant functions and from \eqref{eq:time_discretized_equation_rothe_method}, we are able to write
    \begin{equation}\label{eq:discretized_variational_equation}
        \int_0^t{\int_{\Omega}{a^\varepsilon_n \varphi + \beta^\varepsilon(u^\varepsilon_n)\varphi-g\varphi}\,dx}\,d\tau+\int_0^t\int_{\R^d}{\left(A D^s u^\varepsilon_n\cdot D^s \varphi + \nu B D^s v^\varepsilon_n \cdot D^s\varphi\right)}\,dx\,d\tau=0,
    \end{equation}
    for every $t\in[0,T]$ and every $\varphi\in \mathcal{V}_s$. With this identity we can apply integration by parts,
    \begin{multline*}
        \int_{\Omega}{v^\varepsilon_n(t)\varphi(t)}\,dx-\int_\Omega{w_1 \varphi(0)}\,dx-\int_0^t{\int_\Omega{v^\varepsilon_n \dot{\varphi}}\,dx}\,d\tau+\int_0^t{\int_{\Omega}{\beta^\varepsilon(u^\varepsilon_n)\varphi}\,dx}\,d\tau\\
        +\int_0^t{\int_{\R^d}{\left(A D^s u^\varepsilon_n\cdot D^s \varphi + \nu B D^s v^\varepsilon_n \cdot D^s \varphi\right)}\,dx}\,d\tau=\int_0^t{\int_{\Omega}{g\varphi}\,dx}\,d\tau,
    \end{multline*}
    and then pass to the limit to get
    \begin{equation}\label{eq:regularized_problem_without_identification}
        \begin{aligned}
            &\int_{\R^d}{\dot{u}^\varepsilon(t) \varphi(t)}\,dx-\int_{\Omega}{w_1 \varphi(0)}\,dx-\int_0^t{\int_\Omega{\dot{u}^\varepsilon \dot{\varphi}}\,dx}\,d\tau+\int_0^t{\int_{\Omega}{\xi^\varepsilon \varphi}\,dx}\,d\tau\\
            &\qquad\qquad\qquad\qquad\qquad\qquad+\int_0^t{\int_{\R^d}{\left(A D^s u^\varepsilon\cdot D^s \varphi + \nu B D^s \dot{u}^\varepsilon \cdot D^s \varphi\right)}\,dx}\,d\tau=\int_0^t{\int_{\Omega}{g\varphi}\,dx}\,d\tau.
        \end{aligned}
    \end{equation}
    
    Now we need to check that  $\xi^\varepsilon=\beta^\varepsilon(u^\varepsilon)$ in $L^2(0,T;L^2(\Omega))$. In fact, applying the appropriate limits to \eqref{eq:discretized_variational_equation} and using the fact that $B$ is symmetric, we get that
    \begin{align*}
        &\limsup_{h\to 0}\int_0^t{\int_{\Omega}{\beta^\varepsilon(u^\varepsilon_n)u^\varepsilon_n}\,dx}\,d\tau\\
        &\qquad\qquad\leq \int_0^t{\int_\Omega{gu^\varepsilon-\ddot{u}^\varepsilon u^\varepsilon}\,dx}\,d\tau-\liminf_{h\to 0}{\int_0^t\int_{\R^d}{\left(A D^s u^\varepsilon_n\cdot D^s u^\varepsilon_n + \nu B D^s v^\varepsilon_n \cdot D^s u^\varepsilon_n\right)}\,dx\,d\tau}\\
        &\qquad\qquad = \int_0^t{\int_\Omega{gu^\varepsilon-\ddot{u}^\varepsilon u^\varepsilon}\,dx}\,d\tau-\liminf_{h\to 0}{\int_0^t\int_{\R^d}{\left(A_\mathrm{sym} D^s u^\varepsilon_n\cdot D^s u^\varepsilon_n + \nu B D^s v^\varepsilon_n \cdot D^s u^\varepsilon_n\right)}\,dx\,d\tau}\\
        &\qquad\qquad \leq \int_0^t{\int_\Omega{gu^\varepsilon-\ddot{u}^\varepsilon u^\varepsilon}\,dx}\,d\tau-\liminf_{h\to 0}{\int_0^t\int_{\R^d}{\left(A_\mathrm{sym} D^s u^\varepsilon\cdot D^s u^\varepsilon + \nu B D^s \dot{u}^\varepsilon \cdot D^s u^\varepsilon\right)}\,dx\,d\tau}\\
        &\qquad\qquad \leq \int_0^t{\int_\Omega{gu^\varepsilon-\ddot{u}^\varepsilon u^\varepsilon}\,dx}\,d\tau-\liminf_{h\to 0}{\int_0^t\int_{\R^d}{\left(A D^s u^\varepsilon\cdot D^s u^\varepsilon + \nu B D^s \dot{u}^\varepsilon \cdot D^s u^\varepsilon\right)}\,dx\,d\tau}\\
        &\qquad\qquad =\int_0^t{\int_{\Omega}{\xi^\varepsilon u^\varepsilon_n}\,dx}\,d\tau,
    \end{align*}
    which means that $\xi^\varepsilon= \partial\mathcal{J}^\varepsilon(u^\varepsilon)=\beta^\varepsilon(u^\varepsilon)$, since $j^\varepsilon$ is convex and differentiable.

    Finally, if we test \eqref{eq:regularized_problem_without_identification} with $\varphi=\dot{u}^\varepsilon$ and use the identification $\xi^\varepsilon=\beta^\varepsilon(u^\varepsilon)$, we obtain \eqref{eq:energy_equality_varepsilon}.
\end{proof}

\begin{remark}
    Notice that, since $\xi^\varepsilon=\beta^\varepsilon(u^\varepsilon)$ is in $L^2(0,T; L^2(\Omega))$ we can recover \eqref{eq:regularized_problem_equation_form} in the sense of distributions by testing \eqref{eq:regularized_problem} with functions $w\in C^\infty_c(Q_T)$ and then applying integration by parts. We can also take a time dependent force $g=g(x,t)$ in $L^2(0,T; L^2(\Omega))$ with a simple adaptation of Rothe's method.
\end{remark}

\begin{proposition}
    Let $s\in(\sigma,1]$ with $\sigma>0$,  $\beta=\partial I_{[0,+\infty]}$, which corresponds to the problem \eqref{eq:problem} with zero obstacle ($u\geq 0$),  $\varphi\in H^s_0(\Omega)$, $\varphi\geq0$ and $u^\varepsilon$ be the solution for the penalized viscous problem obtained in Theorem \ref{thm:existence_approximated_solutions}. Then the functionals $F^\varepsilon:[0,T]\to \R$, defined by
    \begin{equation}
        F^\varepsilon(t)=\int_\Omega{\dot{u}^\varepsilon\varphi}\,dx,
    \end{equation}
    are uniformly bounded in $BV(0,T)$ with $\mathrm{var}_0^T(F^\varepsilon)\leq C$ independent of $\varepsilon$, $\nu$ and $s$.
\end{proposition}
\begin{proof}
    Let us define the functions $F^\varepsilon_n:[0,T]\to\R$ as
    \begin{equation}
        F^\varepsilon_n(t)=\int_\Omega{V^\varepsilon_n(t)\varphi}\,dx,
    \end{equation}
    where $V^\varepsilon_n$ is the piecewise affine function as defined in \eqref{eq:piecewise_affine_dscretization}. One can check easily from the uniform boundeness from \eqref{eq:piecewise_affine_dscretization} and \eqref{eq:estimate_1} that $V^\varepsilon_n$ in $L^\infty(0,T;L^2(\Omega))$ and consequently also $F^\varepsilon_n\in L^1(0,T)$ uniformly in $\varepsilon$ and $n$.

    Now, we have to estimate the total variation of $F^\varepsilon_n$
    \begin{multline*}
        \left|\int_\Omega{\frac{v^\varepsilon_{n,j}-v^\varepsilon_{n,j-1}}{h}\varphi}\,dx\right|-\int_\Omega{\frac{v^\varepsilon_{n,j}-v^\varepsilon_{n,j-1}}{h}\varphi}\,dx\\
        \leq \left|\int_{\R^d}{AD^s u^\varepsilon_{n,j}\cdot D^s \varphi}\,dx\right|+\left|\int_{\R^d}{\nu BD^s v^\varepsilon_{n,j}\cdot D^s \varphi}\,dx\right|+\left|\int_\Omega{g\varphi\,dx}\,dx\right|.
    \end{multline*}
    Consequently,
    \begin{align*}
        \sum_{j=1}^n{\left|\int_\Omega{(v^\varepsilon_{n,j}-v^\varepsilon_{n,j-1})\varphi}\,dx\right|}
        &\leq \int_\Omega{v^\varepsilon_{n,n}\varphi}\,dx-\int_\Omega{w_1}\varphi\,dx+\left|\int_\Omega{g\varphi\,dx}\,dx\right|\\
        &\qquad\qquad\qquad\qquad+2h\sum_{j=1}^n{\left(\left|\int_{\R^d}{AD^s u^\varepsilon_{n,j} D^s\varphi}\,dx\right|+\left|\int_{\R^d}{\nu BD^s v^\varepsilon_{n,j} D^s\varphi}\,dx\right|\right)}\\
        &\leq\|v^\varepsilon_{n,n}\|_{L^2(\Omega)}\|\varphi\|_{L^2(\Omega)}+\|w_1\|_{L^2(\Omega)}\|\varphi\|_{L^2(\Omega)}+\|g\|_{L^2(\Omega)}\|\varphi\|_{L^2(\Omega)}\\
        &\qquad\qquad\qquad\qquad +2h\|\varphi\|_{H^s_0(\Omega)}\sum_{j=1}^n\left(\|u^\varepsilon_{n,i}\|_{H^s_0(\Omega)}+\|v^\varepsilon_{n,i}\|_{H^s_0(\Omega)}\right)\\
        &\leq C\|\varphi\|_{H^s_0(\Omega)},
    \end{align*}
    with $C$ independent of $n$ and $\varepsilon$. This means that the function $F^\varepsilon_n$ is uniformly $BV(0,T)$ with respect to $n$. Consequently, by applying Lemma \ref{lemma:limit_bv_functions_and_right_limit} to $\{F^\varepsilon_n\}$, we deduce that $F^\varepsilon$ is in $BV(0,T)$. Moreover, since $\mathrm{Var}_0^T(F^\varepsilon)<C$ with $C$ independent of $\varepsilon$, then we can say that the family of functions $F^\varepsilon$ is uniformly bounded in $BV(0,T)$.
\end{proof}

\section{A priori estimates}\label{sec:apriori_estimates}

\begin{proposition}\label{prop:estimates_for_V}
    For every $0<s\leq 1$ and $\nu>0$, the solutions $u^\varepsilon$ of \eqref{eq:regularized_problem} satisfy the estimate
    \begin{equation}\label{eq:uniform_estimate_for_V}
        \frac{1}{2}\|\dot{u}^\varepsilon(t)\|^2_{L^2(\Omega)}+\frac{a_*}{2}\|D^s u^\varepsilon(t)\|^2_{L^2(\R^d;\R^d)}+\nu b_*\|D^s \dot{u}^\varepsilon\|^2_{L^2(0,T;L^2(\R^d;\R^d))}+J^\varepsilon(u^\varepsilon)\leq C.
    \end{equation}
    for $C>0$ independent of $\varepsilon\in(0,1)$ and $\nu$.
\end{proposition}
\begin{proof}
    From the energy identity \eqref{eq:energy_equality_varepsilon} and from the properties of $A$ and $B$, we have the estimate
    \begin{align*}
        &\frac{1}{2}\left(\|\dot{u}^\varepsilon(t)\|^2_{L^2(\Omega)}+a_*\|D^s u^\varepsilon\|^2_{L^2(\R^d;\R^d)}\right)+J(u^\varepsilon(t))+\nu b_*\|D^s\dot{u}^\varepsilon\|^2_{L^2(0,T;L^2(\R^d;\R^d))}\\
        &\qquad\qquad\leq \frac{1}{2}\left(\|\dot{u}^\varepsilon(t)\|^2_{L^2(\Omega)}+\int_{\R^d}{A D^s u^\varepsilon\cdot D^s u^\varepsilon}\,dx\right)+J(u^\varepsilon(t))+\nu \int_0^t{\int_{\R^d}{BD^s \dot{u}^\varepsilon\cdot D^s \dot{u}^\varepsilon}\,dx}\,dt\\
        &\qquad\qquad=\frac{1}{2}\left(\|w_1\|^2_{L^2(\Omega)}+\int_{\R^d}{A D^s w_0\cdot D^s w_0}\,dx\right)+J^\varepsilon(w_0)+\int_0^t{\int_\Omega{g\dot{u}^\varepsilon}\,dx}\,dt\\
        &\qquad\qquad\leq \frac{1}{2}\|w_1\|^2_{L^2(\Omega)}+\frac{a^*}{2}\|D^s w_0\|^2_{L^2(\R^d;\R^d)}+J^\varepsilon(w_0)+\|g\|^2_{L^2(\Omega)}+\frac{1}{4}\|\dot{u}^\varepsilon\|^2_{L^2(0,T; L^2(\Omega))}\\
        &\qquad\qquad\leq \frac{1}{2}\|w_1\|^2_{L^2(\Omega)}+\frac{a^*}{2}\|D^s w_0\|^2_{L^2(\R^d;\R^d)}+J(w_0)+\|g\|^2_{L^2(\Omega)}+\frac{1}{4}\|\dot{u}^\varepsilon\|^2_{L^2(0,T; L^2(\Omega))}.
    \end{align*}
    Applying Gronwall's lemma, we derive that
    \begin{equation*}
        \|\dot{u}^\varepsilon(t)\|_{L^2(\Omega)}\leq C
    \end{equation*}
     This uniform estimate together with the previous one conclude the proof.
\end{proof}

\begin{remark}\label{rem:independence_of_s}
    If $w_0\in H^1_0(\Omega)$ and $1\geq s>\sigma>0$, then we can also consider the constant $C$ obtained in \eqref{eq:uniform_estimate_for_V} independent of $s$, but dependent of $\sigma$. This is due to the fact that from Proposition 2.7 of \cite{bellido2021Gamma}
    \begin{equation}
        \|D^s w_0\|_{L^2(\R^d;\R^d)}\leq \frac{C}{s}\|w_0\|_{H^1_0(\Omega)}\leq \frac{C}{\sigma}\|w_0\|_{H^1_0(\Omega)}.
    \end{equation}
\end{remark}

\begin{lemma}\label{lemma:auxiliary_lemma}
    There exist constants $c_1>0$, $c_2\geq 0$ independent of $\varepsilon\in(0,1)$ such that
    \begin{equation}
        c_1|\beta^\varepsilon(r)|\leq \beta^\varepsilon(r) r+c_2,\quad\forall r\in\R.
    \end{equation}
\end{lemma}
\begin{proof}
    It is easy to check that outside a sufficiently big neighborhood $N\subset\R$ of $0$ there exists $c_1>0$ such that
    \begin{equation*}
        c_1|\beta^\varepsilon(r)|\leq \beta^\varepsilon(r)r.
    \end{equation*}
    Taking
    \begin{equation*}
        c_2=\max\{\left|c_1|\beta^\varepsilon(r)|-\beta^\varepsilon(r)r\right|:\, \varepsilon\in(0,1), r\in D(\beta)\},
    \end{equation*}
    we conclude the proof.
\end{proof}

\begin{proposition}
    For $t \in (0, T]$ denote $Q_t = [0, t] \times \Omega$. Then there exists a constant $C > 0$ independent of $t \in (0, T]$, $\varepsilon \in (0, 1)$ and of $\nu> 0$ such that
    \begin{equation}
        \iint_{Q_t}{|\beta^\varepsilon(u^\varepsilon)|}\,dx\,dt\leq C.
    \end{equation}
    Moreover, in the case where $w_0\in H^1_0(\Omega)$ and $1\geq s>\sigma>0$, the previous constant $C$ can also be independent of $s$, but still dependent on $\sigma$.
\end{proposition}
\begin{proof}
    Testing \eqref{eq:regularized_problem} with $\varphi=u^\varepsilon$ we get
    \begin{align*}
        &\int_0^t{\int_\Omega{\beta^\varepsilon(u^\varepsilon)u^\varepsilon}\,dx}\,d\tau
        +\alpha \|D^s u^\varepsilon\|_{L^2(0,t; L^2(\Omega))}^2\\
        &\qquad\qquad\leq \int_0^t{\int_\Omega{ \beta^\varepsilon(u^\varepsilon)u^\varepsilon}\,dx}\,d\tau+\int_0^t{\int_{\R^d}{A D^s u^\varepsilon\cdot D^s u^\varepsilon}\,dx}\,d\tau\\
        &\qquad\qquad=\int_\Omega{w_1 w_0}\,dx+\int_0^t{\int_\Omega{|\dot{u}^\varepsilon|^2+gu^\varepsilon}\,dx}\,d\tau-\nu \int_0^t{\int_{\R^d}{B D^s \dot{u}^\varepsilon \cdot D^s u^\varepsilon}\,dx}\,d\tau-\int_\Omega{\dot{u}^\varepsilon(t) u^\varepsilon(t)}\,dx\\
        &\qquad\qquad\leq \|w_1\|_{L^2(\Omega)}\|w_0\|_{L^2(\Omega)}+T\|\dot{u}^\varepsilon\|_{L^\infty(0,T;L^2(\Omega))}^2+T\|g\|_{L^2(\Omega)}\|u^\varepsilon\|_{L^\infty(0,T; L^2(\Omega))}\\
        &\qquad\qquad\qquad\qquad\qquad+\nu b^*\|D^s \dot{u}^\varepsilon\|_{L^2(0,T;L^2(\R^d;\R^d))} \|D^s u^\varepsilon\|_{L^2(0,T;L^2(\R^d;\R^d))}+\|\dot{u}^\varepsilon(T)\|_{L^2(\Omega)}\|u^\varepsilon(T)\|_{L^2(\Omega)},
    \end{align*}
    which is bounded by a constant $C'$ due to Proposition \ref{prop:estimates_for_V}. Consequently, by applying Lemma \ref{lemma:auxiliary_lemma}, we conclude that
    \begin{equation*}
        \iint_{Q_t}{|\beta^\varepsilon(u^\varepsilon)|}\,dx\,dt\leq \frac{1}{c_1}\iint_{Q_t}{\left(\beta^\varepsilon(u^\varepsilon)u^\varepsilon+c_2\right)}\,dx\,dt\leq \frac{1}{c_1}\iint_{Q_t}{\beta^\varepsilon(u^\varepsilon)u^\varepsilon}\,dx\,dt+\frac{c_2}{c_1}|\Omega|T\leq C
    \end{equation*}
\end{proof}

\begin{proposition}\label{prop:uniform_estimate_beta}
    Let $t \in (0, T]$ and $s \in (0, 1]$ be given. Then the sequence $\beta^\varepsilon(u^\varepsilon)$ is uniformly bounded in $\mathcal{V}'_{s,t}$.
\end{proposition}
\begin{proof}
    This is a simple consequence of \eqref{eq:regularized_problem} and of the estimate \eqref{eq:uniform_estimate_for_V}. In fact, for any $w\in\mathcal{V}_s$, we get that
    \begin{align*}
        \langle \beta^\varepsilon(u^\varepsilon), \varphi\rangle_{\mathcal{V}'_{s,t}\times\mathcal{V}_{s,t}}
        &=\int_\Omega{w_1 \varphi(0)}\,dx
        -\int_\Omega{\dot{u}(t) \varphi(t)}\,dx
        +\int_0^t{\int_{\Omega}{\dot{u}^\varepsilon \dot{\varphi}+g\varphi}\,dx}\,d\tau\\
        &\qquad\qquad-\int_0^t{\int_{\R^d}{A D^s u^\varepsilon\cdot D^s \varphi}\,dx}\,d\tau 
        - \nu \int_0^t{\int_{\R^d}{B D^s \dot{u}^\varepsilon \cdot D^s \varphi}\,dx}\,d\tau\\
        &\leq \|w_1\|_{L^2(\Omega)}\|\varphi(0)\|_{L^2(\Omega)}
        +\|\dot{u}(t)\|_{L^2(\Omega)}\|\varphi(t)\|_{L^2(\Omega)}
        +\|\dot{u}\|_{L^2(0,T;L^2(\Omega))}\|\dot{\varphi}\|_{L^2(0,T;L^2(\Omega))}\\
        &\qquad\qquad+\|g\|_{L^2(\Omega)}\|\varphi\|_{L^2(0,T;L^2(\Omega))}
        +a^*\|D^s u^\varepsilon\|_{L^2(0,t;L^2(\Omega))}\|D^s \varphi\|_{L^2(0,T;L^2(\Omega))}\\
        &\qquad\qquad\qquad\qquad+\nu b^*\|D^s\dot{u}^\varepsilon\|_{L^2(0,T;L^2(\Omega))}\|D^s \varphi\|_{L^2(0,T;L^2(\Omega))}\\
        &\leq C\|\varphi\|_{\mathcal{V}_{s,t}}.
    \end{align*}
\end{proof}

\begin{remark}
    If $s>\sigma>0$ and $w_0\in H^1_0(\Omega)$, then there exists a constant $C>0$ independent of $s$, but possibly dependent on $\sigma$ such that $\|\beta^\varepsilon(u^\varepsilon)\|_{\mathcal{V}'_s}\leq C$. This is a consequence of Remark \ref{rem:independence_of_s}.
\end{remark}

\begin{proposition}\label{prop:estimate_for_X}
    Let $k$ be a sufficiently big number depending on $d$ such that $L^1(\Omega)\subset H^{-k}$, $H^{-s}(\Omega)\subset H^{-k}(\Omega)$ for all $s\in(0,1)$ with the embeddings being continuous and compact. There exists a constant $C>0$ independent of $\varepsilon\in(0,1)$ and $\nu$ such that
    \begin{equation}
        \|\ddot{u}^\varepsilon\|_{L^1(0,T; H^{-k})}\leq C.
    \end{equation}
\end{proposition}
\begin{proof}
    Having in mind the previous propositions of this section, we observe that this result is just a comparison argument in the equation \eqref{eq:regularized_problem_equation_form} from the point of view of distributions.
\end{proof}

\section{Passage to the limit as $\varepsilon\to 0$}\label{sec:existence_s_fixed}
With the estimates that we have obtained in the previous section, we are now in position to study the behavior of the functions $u^\varepsilon$ when we take $\varepsilon\to 0$.

\begin{proposition}\label{prop:limit_problem_when_varepsilon_to_0}
    Let $s\in(0, 1]$, $\nu > 0$, $t \in (0, T]$, $w_0 \in H^s_0(\Omega)$, and $w_1\in L^2(\Omega)$ be given. Then there exist a function $u\in L^\infty(0,T; H^s_0(\Omega))\cap H^1(0,T; L^2(\Omega))$ with $u(0)=w_0$, $\dot{u}(0)=w_1$ and a distribution $\xi_t\in\mathcal{V}'_{s,t}$ that satisfy the identity 
    \begin{equation}\label{eq:limit_problem_s_fixed}
        \begin{aligned}
            \int_{\Omega}{\dot{u}(t) \varphi(t)}\,dx
            -\int_\Omega{w_1 \varphi(0)}\,dx
            &-\int_0^t{\int_\Omega{\dot{u} \dot{\varphi}}\,dx}\,d\tau
            +\langle \xi_t, \varphi\rangle_{\mathcal{V}'_{s,t}\times\mathcal{V}_{s,t}}\\
            &\quad +\int_0^t{\int_{\R^d}{\left(A D^s u\cdot D^s \varphi + \nu B D^s \dot{u} \cdot D^s \varphi\right)}\,dx}\,d\tau
            =\int_0^t{\int_{\Omega}{g\varphi}\,dx}\,d\tau,
        \end{aligned}
    \end{equation}
    for all $\varphi\in\mathcal{V}_{s,t}$.
\end{proposition}
\begin{proof}
    Making use of the estimates obtained in Propositions \ref{prop:estimates_for_V} and \ref{prop:uniform_estimate_beta}, as well of Poincaré's inequality for $D^s$ described in Lemma \ref{lemma:poincares_inequality}, we are able to deduce that there exist functions $u$ and $v$, and a distribution $\xi\in\mathcal{V}'_s$ such that
    \begin{align}
        &u^\varepsilon\rightharpoonup u \quad \mbox{ in } H^1(0, T; H^s_0(\Omega)) \label{eq:strong_convergence_u_varepsilon_in_H1_Hs0}\\
        &u^\varepsilon\overset{\ast}{\rightharpoonup} u \quad \mbox{ in } W^{1,\infty}(0, T; L^2(\Omega))\\
        &u^\varepsilon(t)\rightharpoonup u(t) \quad \mbox{ in } H^s_0(\Omega) \mbox{ for all } t\in[0,T],
        \label{eq:weak_convergence_u_varepsilon_in_Hs0}\\
        &u^\varepsilon(t)\to u(t) \quad \mbox{ in } L^2(\Omega) \mbox{ for all } t\in[0,T],
        \label{eq:strong_convergence_u_varespilon_t_in_L2}\\
        &\dot{u}^\varepsilon(t)\rightharpoonup v(t) \quad \mbox{ in } L^2(\Omega) \mbox{ for all } t\in[0,T], \label{eq:weak_convergence_dot_u_varepsilon_in_L2}\\
        &\beta^\varepsilon(u^\varepsilon)\rightharpoonup \xi_t \quad \mbox{ in } \mathcal{V}'_{s,t},
    \end{align}
    up to a subsequence. Moreover, we know that $\dot{u}^\varepsilon\in L^2(0,T; H^s_0(\Omega))\cap W^{1,1}(0,T;H^{-k}(\Omega))$. This allow us to apply Aubin-Lions lemma and deduce that
    \begin{equation}\label{eq:strong_convergence_aubin_lions_u_varepsilon}
        \dot{u}^\varepsilon\to\dot{u}_s \mbox{ in } L^2(0,T;L^2(\Omega)).
    \end{equation}
    At the same time, since $\dot{u}^\varepsilon\in W^{1,1}(0,T;H^{-k}(\Omega))\subset BV(0,T; H^{-k}(\Omega))$ uniformly, then by the generalized Helly's selection lemma, Lemma \eqref{lemma:generalized_helly_selection}, we can say that 
    \begin{equation}
        \dot{u}^\varepsilon(t)\overset{\ast}{\rightharpoonup} \dot{u}(t) \mbox{ in } H^{-k}(\Omega) \mbox{ for all } t\in[0,T].
    \end{equation}
    This last limit in $H^{-k}(\Omega)$ can then be used to identify $v(t)$ with $\dot{u}(t)$ in $L^2(\Omega)$ for all $t\in [0,T]$.

    All the above convergence results allow us to pass to the limit in \eqref{eq:regularized_problem} and obtain \eqref{eq:limit_problem_s_fixed}.
\end{proof}

\begin{theorem}\label{thm:existence_weak_solutions}
    Let $s\in(0,1]$, $\nu>0$. Then the pair $(u,\xi_T)$ with $u\in L^\infty(0,T; H^s_0(\Omega))\cap H^1(0,T; L^2(\Omega))$ and $\xi_t\in\mathcal{V}'_{s,t}$ obtained in the previous proposition is a weak solution of \eqref{eq:problem}.
\end{theorem}
\begin{proof}
    Thanks to Proposition \ref{prop:limit_problem_when_varepsilon_to_0}, we only need to check that  $\xi_t\in\beta_{s,t}(u)$. To do that, we start by observing that from the limits that were obtained in the proof of Proposition \ref{prop:limit_problem_when_varepsilon_to_0} and from the definition of $\beta_s$, we only need to prove that
    \begin{equation}
        \limsup_{\varepsilon\to 0}{\langle \beta^\varepsilon(u^\varepsilon), u^\varepsilon\rangle_{\mathcal{V}'_{s,t}\times\mathcal{V}_{s,t}}}\leq \langle \xi_s, u\rangle_{\mathcal{V}'_{s,t}\times\mathcal{V}_{s,t}}.
    \end{equation}
    But if we test \eqref{eq:regularized_problem} with $\varphi=u^\varepsilon$, we get
    \begin{equation}\label{eq:inequality_for_limsup}
        \begin{aligned}
            \limsup_{\varepsilon\to 0}{\langle \beta^\varepsilon(u^\varepsilon), u^\varepsilon\rangle_{\mathcal{V}'_{s,t}\times\mathcal{V}_{s,t}}}
            &\leq \lim_{\varepsilon\to 0}{\|\dot{u}^\varepsilon\|^2_{L^2(0,t; L^2(\Omega))}}
            -\lim_{\varepsilon\to 0}{\int_\Omega{\dot{u}^\varepsilon(t)u^\varepsilon(t)}\,dx}
            +\int_\Omega{w_1 w_0}\,dx
            \\
            &+\frac{1}{2}\int_{\R^d}{\nu B D^sw_0\cdot D^s w_0}\,dx
            -\frac{1}{2}\liminf_{\varepsilon\to 0}{\int_{\R^d}{\nu B D^su^\varepsilon(t)\cdot D^s u^\varepsilon(t)}\,dx}
            \\
            &-\liminf_{\varepsilon\to 0}{\int_0^t{\int_\Omega{A_{\mathrm{sym}}D^s u^\varepsilon\cdot D^s u^\varepsilon}\,dx}\,d\tau}
            +\lim_{\varepsilon\to 0}{\int_0^t{\int_\Omega{g u^\varepsilon}\,dx}\,d\tau}\\
            &= \langle \xi_t, u\rangle_{\mathcal{V}'_{s,t}\times\mathcal{V}_{s,t}}
        \end{aligned}
    \end{equation}
    thanks to $B$ being symmetric, \eqref{eq:limit_problem_s_fixed}, \eqref{eq:strong_convergence_aubin_lions_u_varepsilon}, \eqref{eq:strong_convergence_u_varespilon_t_in_L2} and  \eqref{eq:weak_convergence_dot_u_varepsilon_in_L2} with $v(t)$ identified with $\dot{u}(t)$. The fact that $\xi_t$ corresponding to the restriction of the distribution $\xi_T$ in the sense of \eqref{eq:restriction_distribution} follows easily by testing \eqref{eq:limit_problem_s_fixed} with functions $\varphi\in\overline{\mathcal{V}}_{s,t}$.
\end{proof}

\begin{remark}[Energy inequality]
    Since $J^\varepsilon = \int_\Omega{j^\varepsilon(v)}\,dx$ converges to $J=\int_\Omega{j(v)}\,dx$ in the sense of Mosco in $L^2(\Omega)$ \cite[Thm.~3.20]{attouch1984Variational}, then
    \begin{equation}\label{eq:energy_inequality_for_s}
        J(u(t))=\int_\Omega{j(u(t))}\,dx\leq \liminf_{\varepsilon\to 0}{\int_\Omega{j^\varepsilon(u^\varepsilon(t))}\,dx}=\liminf_{\varepsilon\to 0}{J^\varepsilon(u^\varepsilon(t))}.
    \end{equation}
    Then, by applying the $\liminf$ as $\varepsilon\to 0$ to \eqref{eq:energy_equality_varepsilon}, and using the limits \eqref{eq:strong_convergence_u_varepsilon_in_H1_Hs0}, \eqref{eq:weak_convergence_u_varepsilon_in_Hs0} and \eqref{eq:strong_convergence_u_varespilon_t_in_L2}, we deduce that
    \begin{equation}
        \mathcal{E}(u_s(t),\dot{u}_s(t))+\int_0^t{\int_{\R^d}{\nu BD\dot{u}_s\cdot\dot{u}_s}\,dx}\,d\tau\leq \mathcal{E}(w_0,w_1)+\int_0^t{\int_\Omega{g\dot{u}_s(t)}\,dx}\,d\tau.
    \end{equation}
    where $\mathcal{E}$ is defined by
    \begin{equation}
        \mathcal{E}(v,w)=\frac{1}{2}\left(\|w\|^2_{L^2(\Omega)}+\int_{\R^d}{A D^s v\cdot D^s v}\,dx\right)+\int_\Omega{j(v)}\,dx.
    \end{equation}
\end{remark}

\begin{lemma}\label{lemma:convergence_F_varepsilon_to_F_s}
    Let $s\in(0,1]$, $\nu>0$ and $0\leq \psi \in H^s_0(\Omega)$. If the function $u$ is the weak solution to \eqref{eq:problem} obtained in Theorem \ref{thm:existence_weak_solutions} for $\beta=\partial I_{[0,+\infty)}$, then the function $F:[0,T]\to \R$ defined as
    \begin{equation}
        F(t)=\int_\Omega{\dot{u}(t)\varphi}\,dx,
    \end{equation}
    is in $BV(0,T)$, in particular $\mathrm{Var}_0^T(F)\leq C$ with $C$ independent of $\nu$ and $s$, and there exists a function $\dot{u}_s(0^+)\in L^2(\Omega)\cap H^{-k}(\Omega)$ such that
    $\dot{u}_s(t)\to \dot{u}_s(0^+)$ in $H^{-k}(\Omega)$ as $t\searrow 0$ and 
    \begin{equation}
        \lim_{t\searrow 0}F(t)=\int_{\Omega}{\dot{u}_s(0^+)\varphi}\,dx.
    \end{equation}
\end{lemma}
\begin{proof}
    This is a simple application of Lemma \ref{lemma:limit_bv_functions_and_right_limit}, which is possible due to the fact that $F^\varepsilon$ is in $BV(0,T)$ uniformly with respect to $\varepsilon$, $\|\dot{u}^\varepsilon(t)\|_{L^2(\Omega)}\leq C$ uniformly in $\varepsilon$ and $\dot{u}^\varepsilon\overset{\ast}{\rightharpoonup} \dot{u}_s$ in $L^\infty(0,T;L^2(\Omega))$.
\end{proof}

\begin{proposition}\label{prop:weak_is_very_weak}
    Let $s\in(0,1]$ and $\nu>0$. If a function $u$ is a weak solution to \eqref{eq:problem} obtained in Theorem \ref{thm:existence_weak_solutions} for $\beta=\partial I_{[0,+\infty)}$, then $u$ is also a very weak solution of the same problem.
\end{proposition}
\begin{proof}
    We start by observing that if we test \eqref{eq:variational_inequality} with $v=\varphi+u$ where $\varphi\in H^1(0,T; L^2(\Omega))\cap L^2(0,T; H^s_0(\Omega))$ is such that $\varphi\geq 0$ and $\mathrm{supp}(\varphi(\cdot,x))\subset [0,T)$, then we get that
    \begin{equation}\label{eq:signal_inequality_s_fixed}
            \int_0^T{\int_{\R^d}{\left(A D^s u\cdot D^s\varphi + \nu B D^s \dot{u} \cdot D^s\varphi\right)}\,dx}\,dt
            \geq \int_\Omega{w_1(\varphi(0)-w_0)}\,dx
            +\int_0^T{\int_\Omega{\dot{u} \dot{\varphi}}\,dx}\,dt
            +\int_0^T{\int_{\Omega}{g\varphi}\,dx}\,dt.
    \end{equation}
    Now we have to check that the initial conditions satisfy \eqref{eq:initial_conditions_very_weak_solutions}. For that, we now test \eqref{eq:variational_inequality} with the function $v$ that is equal to a function $\psi\in H^s_0(\Omega)$ at all time $t$, i.e.,  $v(\cdot,t)=\psi\in H^s_0(\Omega)$, and deduce the estimates
    \begin{equation}\label{eq:estimate_nu_to_0_g}
        \left|\int_0^t{\int_{\Omega}{g(\psi-u)}\,dx}\,dt\right|\leq t\|g\|_{L^2(\Omega)}\left(\|\psi\|_{L^2(\Omega)}+\|u\|_{L^\infty(0,t;L^2(\Omega))}\right)\leq Ct^{1/2},
    \end{equation}
    \begin{equation}\label{eq:estimate_nu_to_0_velocity}
        \int_0^t{\int_\Omega{|\dot{u}|^2}\,dx}\,dt\leq t\|\dot{u}\|^2_{L^\infty(0,T;L^2(\Omega))}\leq Ct,
    \end{equation}
    \begin{multline}\label{eq:estimate_nu_to_0_A}
        \int_0^t{\int_{\R^d}{AD^s u \cdot D^s(\psi-u)}\,dx}\,dt\\
        \leq ta^*\|D^s u\|_{L^\infty(0,t;L^2(\R^d;\R^d))}\left(\|D^s \psi\|_{L^2(\R^d;\R^d)}+\|D^s u\|_{L^\infty(0,t;L^2(\R^d;\R^d))}\right)\leq Ct,
    \end{multline}
    and
    \begin{multline}\label{eq:estimate_nu_to_0_B}
        \int_0^t{\int_{\R^d}{\nu BD^s \dot{u} \cdot D^s(\psi-u)}\,dx}\,dt\\
        \leq \nu b^*\|D^s \dot{u}\|_{L^2(0,T;L^2(\Omega))}\left(t^{1/2}\|D^s \psi\|_{L^\infty(0,t;L^2(\Omega))}+t^{1/2}\|D^s u\|_{L^\infty(0,T;L^2(\Omega))}\right)\leq Ct^{1/2},
    \end{multline}
    where $C$ is independent of $\nu$ and $s$. 
    Then, from the fact that $u$ is a solution to \eqref{eq:variational_inequality} we get that
    \begin{equation}\label{eq:estimate_for_initial_conditions}
        \int_\Omega{\dot{u}(t)(\psi-u(t))}\,dx-\int_\Omega{w_1(\psi-w_0)}\,dx\geq -Ct-Ct^{1/2}.
    \end{equation}
    Applying Lemma \ref{lemma:convergence_F_varepsilon_to_F_s} as $t\searrow 0$ to the previous inequality, and using the fact that $u$ is continuous in $L^2(\Omega)$, we conclude that
    \begin{equation*}
        \int_\Omega{(\dot{u}(0^+)-w_1)(\psi-w_0)}\,dx\geq 0.
    \end{equation*}
\end{proof}

\section{From viscous to inviscid membranes with zero obstacle}\label{sec:from_viscous_to_inviscid}

\begin{theorem}\label{thm:existence_of_solutions_inviscid_problem}
    Let $s\in(0,1]$. Consider a sequence of functions $\{u_\nu\}$ that correspond to the weak solutions of \eqref{eq:problem} obtained in Theorem \ref{thm:existence_weak_solutions} for $\beta=\partial I_{[0,+\infty)}$ and $\nu>0$. Then there exists a subsequence $\nu\to 0$ and a function $u\in L^\infty(0,T; H^s_0(\Omega))\cap H^1(0,T; L^2(\Omega))$ such that $u_\nu\rightharpoonup u$ in $H^1(0,T; H^s_0(\Omega))$ which is also a very weak solution of the inviscid problem with $\nu=0$ in the sense of Definition \ref{def:vey_weak_solution_inviscid}.
\end{theorem}
\begin{proof}
    Since the estimate \eqref{eq:uniform_estimate_for_V} is independent of $\nu$, we then have, just like in the proof of Proposition \ref{prop:limit_problem_when_varepsilon_to_0} that there exists a function $u\in L^\infty(0,T; H^s_0(\Omega))\cap H^1(0,T; L^2(\Omega))$ such that
\begin{align}
        &u_\nu\rightharpoonup u \quad \mbox{ in } H^1(0, T; L^2(\Omega))\\
        &u_\nu\overset{\ast}{\rightharpoonup} u \quad \mbox{ in } W^{1,\infty}(0, T; L^2(\Omega))\\
        &u_\nu(t)\rightharpoonup u(t) \quad \mbox{ in } H^s_0(\Omega) \mbox{ for all } t\in[0,T],\\
        &u_\nu(t)\to u(t) \quad \mbox{ in } L^2(\Omega) \mbox{ for all } t\in[0,T],\\
        &\dot{u}_\nu(t)\rightharpoonup \dot{u}(t) \quad \mbox{ in } L^2(\Omega) \mbox{ for all } t\in[0,T].
    \end{align}
    Moreover, we also have from \eqref{eq:uniform_estimate_for_V} that
    \begin{equation*}
        \nu\|D^s u_\nu\|^2\leq C
    \end{equation*}
    with $C$ independent of $\nu$, which implies that
    \begin{equation*}
        \nu\int_0^T{\int_\Omega{ B D^s u_\nu\cdot D^s \varphi}\,dx}\,dt\leq \sqrt{\nu}b^*\|D^s u_\nu\|\sqrt{\nu}\|D^s \varphi\|\leq \sqrt{\nu}C.
    \end{equation*}
    On the one hand, using the fact that $u_\nu$ is also a very weak solution, when we take $\nu\to 0$ we get
    \begin{equation}\label{eq:signal_estimate_without_nu}
            -\int_0^T{\int_\Omega{\dot{u} \dot{\varphi}}\,dx}\,dt
            +\int_0^T{\int_{\R^d}{A D^s u\cdot D^s\varphi}\,dx}\,dt
            \geq \int_\Omega{w_1(\varphi(0)-w_0)}\,dx
            +\int_0^T{\int_{\Omega}{g\varphi}\,dx}\,dt.
    \end{equation}
    On the other hand, due to the fact that the estimate is \eqref{eq:estimate_for_initial_conditions} independent of $\nu$, we have that
    \begin{equation*}
        \int_\Omega{\dot{u}(t)( \psi-u(t))}\,dx-\int_\Omega{w_1(\psi-w_0)}\,dx=\lim_{\nu\to 0}{\int_\Omega{(\dot{u}_\nu(t))(\psi-u_\nu(t))}\,dx-\int_\Omega{w_1(\psi-w_0)}\,dx}\geq -Ct-Ct^{1/2},
    \end{equation*}
    for $0\leq \psi\in H^s_0(\Omega)$. Now we just need to make $t\searrow 0$. For that consider the functions $F^\nu:[0,T]\to \R$ defined as
    \begin{equation*}
        F^\nu(t)=\int_\Omega{\dot{u}_\nu(t)\varphi}\,dx,
    \end{equation*}
    as well as the function $F:[0,T]\to \R$ defined as
    \begin{equation*}
        F(t)=\int_\Omega{\dot{u}(t)\varphi}\,dx.
    \end{equation*}
    Since $\mathrm{Var}_0^T(F^\nu)$ is uniformly bounded with respect to $\nu$, and $\|u^\nu(t)\|_{L^2(\Omega)}\leq C$ for all $t\in [0,T]$, then we can apply Lemma \ref{lemma:limit_bv_functions_and_right_limit} and deduce that $F$ is $BV(0,T)$ and that there exists a function $\dot{u}(0^+)\in L^2(\Omega)$ such that
    $\dot{u}(t)\rightharpoonup \dot{u}(0^+)$ in $L^2(\Omega)$ as $t\searrow 0$ and 
    \begin{equation*}
        \lim_{t\searrow 0}F(t)=\int_{\Omega}{\dot{u}(0^+)\varphi}\,dx.
    \end{equation*}
    This then allow us to conclude that
    \begin{equation*}
        \int_\Omega{(\dot{u}(0^+)-w_1)(\psi-w_0)}\,dx\geq 0.
    \end{equation*}
\end{proof}

\section{From fractional to classical as $s\nearrow 1$}\label{sec:existence_s_to_1}
We show the stability of weak and very weak solutions of the viscous, $\nu>0$, and inviscid, $\nu=0$, problems \eqref{eq:problem}, respectively, when the fractional parameter $s$ tends to $1$.

\subsection{Convergence of the weak solutions in the general viscous problem}

\begin{theorem}
    Let $\nu>0$, $w_0\in H^1_0(\Omega)$, $w_1\in L^2(\Omega)$, $\sigma\in(0,1)$ and consider a sequence of values $s\in (\sigma,1)$ such that $s\nearrow 1$. Consider also sequence of functions $\{u_s\}_s$ and a sequence of distributions $\{\xi_s\}$ with each pair $(u_s, \xi_s)$ being weak solutions obtained in Theorem \ref{thm:existence_weak_solutions} for each $s$. Then, we can extract a subsequence from $\{u_s\}$ and $\{\xi_s\}$ such that \begin{equation*}
        u_s\rightharpoonup u\quad \mbox{ in }\quad H^1(0,T; H^\sigma_0(\Omega)), \qquad \xi_s\rightharpoonup \xi\quad \mbox{ in }\quad \mathcal{V}'_{1}
    \end{equation*}
    with $u\in H^1(0,T; H^1_0(\Omega))$ and $\xi\in \beta_1(u)$ being a weak solution of \eqref{eq:problem} for $s=1$.
\end{theorem}
\begin{remark}
    Due to the lack of uniqueness of solutions to \eqref{eq:limit_problem_s_fixed}, we cannot say that the solution $u_1$ to this problem when $s=1$ coincides with the limit $u_1$.
\end{remark}
\begin{proof}
    To construct the function $u$ and the distributions $\xi$ and $\xi_t$ with $t\in(0,T]$, satisfying
    \begin{equation}\label{eq:limit_problem_for_1}
        \begin{aligned}
            \int_\Omega{\dot{u}(t) \varphi(t)}\,dx
            -\int_\Omega{w_1 \varphi(0)}\,dx
            &-\int_0^t{\int_\Omega{\dot{u} \dot{\varphi}}\,dx}\,d\tau
            +\langle \xi_t, \varphi\rangle_{\mathcal{V}'_{1,t}\times\mathcal{V}_{1,t}}\\
            &\qquad+\int_0^t{\int_{\R^d}{\left(A D u\cdot D \varphi + \nu B D \dot{u} \cdot D \varphi\right)}\,dx}\,d\tau
            =\int_0^t{\int_{\Omega}{g\varphi}\,dx}\,d\tau
        \end{aligned}
    \end{equation}
    for all $\varphi\in\mathcal{V}_{1,t}$, we are going to follow the same arguments as those used to prove Proposition \ref{prop:limit_problem_when_varepsilon_to_0}.

    We start by observing that since $w_0\in H^1_0(\Omega)$, then for every $t\in[0,T]$ we have
    \begin{equation}
        \begin{aligned}
            &\frac{1}{2}\|\dot{u}_s(t)\|^2_{L^2(\Omega)}+\frac{a_*}{2}\|D^s u_s(t)\|^2_{L^2(\R^d;\R^d)}+b_*\|D^s \dot{u}_s\|^2_{L^2(0,T;L^2(\R^d;\R^d))}\\
            &\qquad\qquad\leq \liminf_{\varepsilon\to 0}{\left(\frac{1}{2}\|\dot{u}^\varepsilon(t)\|^2_{L^2(\Omega)}+\frac{a_*}{2}\|D^s u^\varepsilon(t)\|^2_{L^2(\R^d;\R^d)}+b_*\|D^s \dot{u}^\varepsilon\|^2_{L^2(0,T;L^2(\R^d;\R^d))}\right)}\leq C.
        \end{aligned}
    \end{equation}
    This then allow us to say that there exists a function $u\in H^1(0,T;H^1_0(\Omega))$
    \begin{align}
        &u_s\rightharpoonup u \quad \mbox{ in } H^1(0, T; H^{\sigma}_0(\Omega))\\
        &D^s u_s(t)\rightharpoonup D u(t) \quad \mbox{ in } L^2(\Omega) \mbox{ for all } t\in[0,T]\label{eq:weak_convergence_Ds_us_in_L2}\\
        & u_s(t)\to u(t) \quad \mbox{ in } L^2(\Omega) \mbox{ for all } t\in[0,T].
    \end{align}
    On the other hand, just like we did in the proof of Proposition \ref{prop:limit_problem_when_varepsilon_to_0}, since $\dot{u}_s(t)$ is uniformly bounded in $L^2(\Omega)$ and $\dot{u}_s$ is uniformly bounded in $W^{1,1}(0,T;H^{-k}(\Omega))\subset BV(0,T;H^{-k}(\Omega))$, we can use Helly's theorem, to conclude that
    \begin{equation}\label{eq:weak_convergence_dot_us_in_L2}
        \dot{u}_s(t)\rightharpoonup \dot{u}(t) \quad \mbox{ in } L^2(\Omega) \mbox{ for all } t\in[0,T]
    \end{equation}
    Moreover, since $\dot{u}_s\in L^2(0,T; H^{\sigma}_0(\Omega))\cap W^{1,1}(0,T;H^{-k}(\Omega))$, we can use Aubin-Lions' theorem and deduce that
    \begin{equation*}
        \dot{u}_s\to\dot{u} \mbox{ in } L^2(0,T; H^\sigma_0(\Omega)).
    \end{equation*}
    In addition, we have a uniform bound for $\xi_s$ in ${\mathcal{V}'_{1,t}}$,
    \begin{equation*}
        \|\xi_{s,t}\|_{\mathcal{V}'_{1,t}}\leq C\|\xi_{s,t}\|_{\mathcal{V}'_{s,t}}\leq \liminf_{\varepsilon\to 0}{C\|\beta^\varepsilon(u^\varepsilon)\|_{\mathcal{V}'_{s,t}}}\leq C'.
    \end{equation*}
    This allow us to say that there exists a distribution $\xi\in \mathcal{V}'_{1,t}$ such that
    \begin{equation*}
        \xi_{s,t}\rightharpoonup\xi_t \mbox{ in } \mathcal{V}'_{1,t}.
    \end{equation*}
    Then, by making use of the Lemma \ref{lemma:strong_convergence_Ds_D_fixed_function} and of all these limits  in \eqref{eq:limit_problem_s_fixed} with the restriction of $\varphi\in\mathcal{V}_{1,t}\subset\mathcal{V}_{s,t}$ we deduce \eqref{eq:limit_problem_for_1}.

    Finally we just have to check that $\xi_t\in\beta_{1,t}(u)$. Since $u_s\to u$ in $L^2(Q_t)$ and $Q_t$ has finite measure in $\R^d\times\R$, we may extract a subsequence such that $u_s\to u$ a.e. in $Q_t$. Then we use the fact that $j$ is lower semicontinuous, that is $j(u)\leq\liminf_{s\to 1}{j(u_s)}$ a.e. in $Q_t$, and we apply Fatou's lemma to deduce
    \begin{equation*}
        \mathcal{J}(u)=\int_0^t{\int_\Omega{j(u)}\,dx}\,d\tau\leq \int_0^t{\int_\Omega{\liminf_{s\to 1}j(u_s)}\,dx}\,d\tau\leq \liminf_{s\to 1}{\int_0^t{\int_\Omega{j(u_s)}\,dx}\,d\tau}=\liminf_{s\to 1}{\mathcal{J}(u_s)}.
    \end{equation*}
    Moreover, since $\mathcal{V}_{1,t}\subset\mathcal{V}_{s,t}$, $\mathcal{V}_{s,t}'\subset\mathcal{V}_{1,t}'$, and $\xi_{s,t}\rightharpoonup\xi_t$ in $\mathcal{V}_{1,t}'$, then for every $\varphi\in \mathcal{V}_{1,t}$ we have
    \begin{equation*}
        \lim_{s\to 1}{\langle \xi_{s,t},\varphi\rangle_{\mathcal{V}'_{s,t}\times\mathcal{V}_{s,t}}}=\lim_{s\to 1}{\langle \xi_{s,t},\varphi\rangle_{\mathcal{V}'_{1,t}\times\mathcal{V}_{1,t}}}=\langle \xi_t,\varphi\rangle_{\mathcal{V}'_{1,t}\times\mathcal{V}_{1,t}}.
    \end{equation*}
    This limit, combined with the lower semicontinuity of $\mathcal{J}$ and the fact $\xi_{s,t}\in\beta_{s,t}(u_s)$, yields
    \begin{equation}
        \begin{aligned}
            \mathcal{J}(u)+\langle\xi_t, \varphi\rangle_{\mathcal{V}'_{1,t}\times\mathcal{V}_{1,t}}+\mathcal{J}(\varphi)
            &\leq \liminf_{s\to 1}{\left(\mathcal{J}(u_s)+\langle\xi_{s,t}, \varphi\rangle_{\mathcal{V}'_{s,t}\times\mathcal{V}_{s,t}}+\mathcal{J}(\varphi)\right)}\\
            &\qquad\qquad\leq \liminf_{s\to 1}{\langle\xi_{s,t}, u_s\rangle_{\mathcal{V}'_{s,t}\times\mathcal{V}_{s,t}}}
            \leq \limsup_{s\to 1}{\langle\xi_{s,t}, u_s\rangle_{\mathcal{V}'_{s,t}\times\mathcal{V}_{s,t}}}.
        \end{aligned}
    \end{equation}
    To conclude the proof we only need to check that
    \begin{equation*}
        \limsup_{s\to 1}{\langle \xi_{s,t}, u_s\rangle_{\mathcal{V}'_{s,t}\times\mathcal{V}_{s,t}}}\leq \langle \xi_t, u\rangle_{\mathcal{V}'_{1,t}\times\mathcal{V}_{1,t}}
    \end{equation*}
    holds, but this follows similarly to \eqref{eq:inequality_for_limsup} as in the proof of Theorem \ref{thm:existence_weak_solutions}.  
\end{proof}

\subsection{Convergence of the very weak solutions in the inviscid obstacle problem}

\begin{theorem}\label{thm:existence_inviscid_solutions}
    Consider a sequence of functions $\{u_s\}$ that correspond to the very weak solutions of the inviscid problem obtained in Theorem \ref{thm:existence_of_solutions_inviscid_problem}. Then there exists a function $u\in L^\infty(0,T; H^1_0(\Omega))\cap H^1(0,T; L^2(\Omega))$ such that $u_s\rightharpoonup u$ in $H^1(0,T; H^1_0(\Omega))$, which is also a very weak solution of the inviscid problem with $s=1$.
\end{theorem}
\begin{proof}
    This is a similar proof to the one of Theorem \ref{thm:existence_of_solutions_inviscid_problem}. In fact, we can use the uniform estimates that are used in the proof of that theorem (recall that these are also uniform with respect to $s$) and Lemma \ref{lemma:bellido_compactness} to get that there exists a function $u\in L^\infty(0,T; H^1_0(\Omega))\cap H^1(0,T; L^2(\Omega))$ such that
    \begin{align}
        &u_s\rightharpoonup u \quad \mbox{ in } H^1(0, T; L^2(\Omega))\\
        &u_s\overset{\ast}{\rightharpoonup} u \quad \mbox{ in } W^{1,\infty}(0, T; L^2(\Omega))\\
        &D^s u_s(t)\rightharpoonup D u_1(t) \quad \mbox{ in } L^2(\Omega) \mbox{ for all } t\in[0,T],\\
        &u_s(t)\to u(t) \quad \mbox{ in } L^2(\Omega) \mbox{ for all } t\in[0,T],\\
        &\dot{u}_s(t)\rightharpoonup \dot{u}(t) \quad \mbox{ in } L^2(\Omega) \mbox{ for all } t\in[0,T].
    \end{align}
    Applying these limits to \eqref{eq:signal_estimate_without_nu}, we deduce that when $s\to 1$ we have
    \begin{equation}
            -\int_0^T{\int_\Omega{\dot{u}_\nu \dot{\varphi}}\,dx}\,dt
            +\int_0^T{\int_{\R^d}{A D u\cdot D^s\varphi}\,dx}\,dt
            \geq (w_1, \varphi(0)-w_0)
            +\int_0^T{\int_{\Omega}{g\varphi}\,dx}\,dt.
    \end{equation}
    Moreover, since due to the fact that the estimate \eqref{eq:estimate_for_initial_conditions} is independent of $s$, we have
    \begin{equation*}
        \int_\Omega{\dot{u}(t)(\psi-u(t))}\,dx-\int_\Omega{w_1(\psi-w_0)}\,dx=\lim_{s\to 1}{\int_\Omega{\dot{u}_s(t)(\psi-u_s(t))}\,dx-\int_\Omega{w_1(\psi-w_0)}\,dx}\geq -Ct-Ct^{1/2},
    \end{equation*}
    for $\psi\in H^s_0(\Omega)$, $\psi\geq0$. Similarly to the proof of Proposition \ref{prop:weak_is_very_weak}, we use the estimates \eqref{eq:estimate_nu_to_0_g}, \eqref{eq:estimate_nu_to_0_velocity} and \eqref{eq:estimate_nu_to_0_A} to get
    \begin{equation*}
        \int_\Omega{\dot{u}_s(t)(\psi-u_s(t))}\,dx-\int_\Omega{w_1(\psi-w_0)}\,dx\geq -Ct-Ct^{1/2}.
    \end{equation*}
    Consider now the functions $F^s:[0,T]\to \R$ defined as
    \begin{equation*}
        F^s(t)=\int_\Omega{\dot{u}_s(t)\varphi}\,dx,
    \end{equation*}
    as well as the function $F:[0,T]\to \R$ defined as
    \begin{equation*}
        F(t)=\int_\Omega{\dot{u}(t)\varphi}\,dx.
    \end{equation*}
    Using exactly the same argument as in the proof of Theorem \ref{thm:existence_of_solutions_inviscid_problem} we conclude, for arbitrary $\psi\in H^1_0(\Omega)$, $\psi\geq0$,
    \begin{equation*}
        \int_\Omega{(\dot{u}(0^+)-w_1)(\psi-w_0)}\,dx\geq 0.
    \end{equation*}
\end{proof}

\begin{remark}
    In \cite{bonafini2021Semilinear}  the authors consider also the homogeneous fractional obstacle problem with a semilinear term $g=W'(u)$, with a suitable potential $W$. As in Theorem \ref{thm:existence_of_solutions_inviscid_problem} we could also obtain a very weak solution to the corresponding semilinear pertubation with the heterogeneous linear operator $\mathscr{A}^s=-D^s\cdot (AD^s\cdot)$, for which a similar stability result as $s\nearrow 1$ as in Theorem \ref{thm:existence_inviscid_solutions} could be obtained.
\end{remark}

\section*{Acknowledgement}

The authors acknowledge the referee's careful reading of the initial manuscript and his suggestions, which allowed the improvement of the final presentation of the work. The authors' research was done under the framework of CMAFcIO, FCT project: UIDB/04561/2020 and UIDP/04561/2020 
and P. M. Campos was supported also by the Portuguese PhD FCT-grant UI/BD/152276/2021.

\end{document}